\documentclass[11pt]{amsart}

\usepackage[a4paper,hmargin=3.5cm,vmargin=4cm]{geometry}
\usepackage{amsfonts,amssymb,amscd,amstext}

\usepackage{fancyhdr} 
\usepackage{times}
\usepackage{enumerate}
\usepackage{titlesec}
\usepackage{mathrsfs}

\titleformat{\section}
{\filcenter\bfseries\large} {\thesection{.}}{0.2cm}{}
\titleformat{\subsection}[runin]
{\bfseries} {\thesubsection{.}}{0.15cm}{}[.]
\titleformat{\subsubsection}[runin]
{\em}{\thesubsubsection{.}}{0.15cm}{}[.]

\usepackage[up,bf]{caption}

\newtheorem{theorem}{Theorem}[section]

\newtheorem{lemma}[theorem]{Lemma}
\newtheorem{corollary}[theorem]{Corollary}

\theoremstyle{definition}
\newtheorem{definition}[theorem]{Definition}
\newtheorem{remark}[theorem]{Remark}

\newtheorem{problem}[theorem]{Problem}

\numberwithin{equation}{section}
\numberwithin{figure}{section}

\def\Pcal{\mathcal{P}}

\def\Ucal{\mathcal{U}}

\def\Ascr{\mathscr{A}}
\def\Cscr{\mathscr{C}}

\def\Oscr{\mathscr{O}}

\def\c{\mathbb{C}}
\def\cp{\mathbb{CP}}
\def\z{\mathbb{Z}}

\def\r{\mathbb{R}}
\def\n{\mathbb{N}}
\def\s{\mathbb{S}}

\def\z{\mathbb{Z}}

\def\pgot{\mathfrak{p}}

\def\Agot{\mathfrak{A}}

\def\Ogot{\mathfrak{O}}

\def\Flux{\mathrm{Flux}}
\def\reg{\mathrm{reg}}

\newcommand\wt{\widetilde}
\newcommand\wh{\widehat}

\newcommand\di{\partial}
\newcommand\dibar{\overline\partial}
\newcommand\hra{\hookrightarrow}

\newcommand\GCMI{\mathrm{GCMI}}
\newcommand\CMI{\mathrm{CMI}}

\begin{document}

\fancyhead[LO]{Embedded minimal surfaces in $\r^n$}
\fancyhead[RE]{A.\ Alarc\'on, F.\ Forstneri\v c, and F.J.\ L\'opez}
\fancyhead[RO,LE]{\thepage}

\thispagestyle{empty}

\vspace*{1cm}
\begin{center}
{\bf\LARGE Embedded minimal surfaces in $\r^n$}

\vspace*{0.5cm}

{\large\bf Antonio Alarc\'on, Franc Forstneri\v c, and Francisco J.\  L\'opez}
\end{center}

\vspace*{1cm}

\begin{quote}
{\small
\noindent {\bf Abstract}\hspace*{0.1cm}
In this paper, we prove that  every conformal minimal immersion of an open Riemann surface into $\r^n$ for $n\ge 5$ 
can be approximated uniformly on compacts by conformal minimal embeddings (see Theorem \ref{th:desing}). 
Furthermore, we show that every open Riemann surface carries a proper conformal minimal embedding into  $\r^5$
(see Theorem \ref{th:proper}). One of our main tools is a Mergelyan approximation theorem for 
conformal minimal immersions to $\r^n$ for any $n\ge 3$ which is also proved in the paper
(see Theorem \ref{th:Mergelyan}).

\vspace*{0.1cm}

\noindent{\bf Keywords}\hspace*{0.1cm} Riemann surfaces, minimal surfaces, conformal minimal embeddings.

\vspace*{0.1cm}

\noindent{\bf MSC (2010):}\hspace*{0.1cm} 53A10; 32B15, 32E30, 32H02.}
\end{quote}


\section{Introduction} 
\label{sec:intro}
One of the central questions in Geometric Analysis is to understand whether an abstract manifold of a certain kind is embeddable 
as a submanifold of a Euclidean space. Notable 
results were obtained, among many others, by Whitney \cite{Wh} in 
differential geometry, Nash \cite{Nash} in Riemannian geometry, Greene and Wu  \cite{GW} in harmonic mapping theory, 
and Remmert \cite{Rem}, Bishop \cite{Bis}, Narasimhan \cite{Nar1,Nar2}, Eliashberg and Gromov \cite{EG} and 
Sch\"urmann \cite{Sch} in complex geometry. 

The same question is highly interesting in the context of {\em minimal submanifolds}, a fundamental subject in differential geometry. 
Two dimensional minimal submanifolds (i.e., {\em minimal surfaces}) in $\r^n$ are especially interesting objects. 
They lie at the intersection of several branches of Mathematics and Physics and enjoy powerful tools coming from 
differential geometry, topology, partial differential equations, and complex analysis; 
see \cite{Osserman} for a classical survey and \cite{MP1,MP2} for more recent ones, among others.

The aim of this paper is to obtain general embedding results for minimal surfaces in $\r^n$ for $n\geq 5$.  
Our first main result is the following.

\begin{theorem}
\label{th:desing}
Let $M$ be an open Riemann surface. If $n\ge 5$ then every conformal minimal immersion 
$u\colon M\to\r^n$ can be approximated uniformly on compacts in $M$ by conformal minimal embeddings. 
The same holds if $M$ is a compact bordered Riemann surface with nonempty boundary and
$u$ is of class $\Cscr^r(M)$ for some $r\in\n$; in such case, the approximation takes place in the $\Cscr^r(M)$ topology.
\end{theorem}

Recall that a conformal immersion $u\colon M\to\r^n$ ($n\geq 3$) of an open Riemann surface
$M$ is minimal if and only if it is harmonic: $\Delta u=0$. An immersion $u\colon M\to\r^n$ is said to be 
an {\em embedding} if $u\colon M\to u(M)$ is a homeomorphism.

More precisely, our proof will show that, for any $n\ge 5$, the set of all conformal minimal embeddings
$M\hra\r^n$ is of the second category in the Fr\'echet space of all conformal minimal immersions
$M\to\r^n$, endowed with the compact-open topology. 

Theorem \ref{th:desing} obviously fails in dimensions $n\le 4$ since transverse self-intersections
are stable in these dimensions. By using the tools of this paper, it can be seen that a
generic conformal minimal immersion $M\to\r^4$ has only simple double points
(normal crossings).

We now come to the following second main result of the paper.

%
%
%
%
\begin{theorem}\label{th:proper}
Every open Riemann surface $M$ carries a  proper  conformal minimal embedding into  $\r^5$.
Furthermore, if $K$ is a compact holomorphically convex set in $M$ and 
$n\geq 5$, then every conformal minimal embedding from a neighborhood of $K$ into $\r^n$ 
can be approximated, uniformly  on $K$, by proper conformal minimal embeddings $M\hra \r^n$.
\end{theorem}

Our methods also provide the control of the {\em flux}  of the conformal minimal embeddings that we construct.
Recall that the flux of a conformal minimal immersion $u=(u_1,\ldots,u_n)\colon M\to\r^n$ is the homomorphism
$H_1(M;\z)\to\r^n$ given by the imaginary periods of the holomorphic $(1,0)$-form $\di u=(\di u_1,\ldots,\di u_n)$
(see \eqref{eq:flux}). In particular, the approximation in Theorem \ref{th:desing} can be done by embeddings 
with the same flux as the original immersion, whereas the proper conformal minimal embeddings in the first assertion 
of Theorem \ref{th:proper} can be found with any given flux; see Theorems \ref{th:desing2} \and \ref{th:proper2}. 

One of our main tools is a Mergelyan approximation theorem for conformal minimal immersions to $\r^n$ for any $n\ge 3$ 
that is also proved in the paper (see Theorem \ref{th:Mergelyan}), extending the result of Alarc\'on and L\'opez  \cite{AL1} 
which applies to $n=3$.  The proof in \cite{AL1} uses the Weierstrass representation of conformal minimal immersions 
$M\to\r^3$, and hence it does not generalize to the case $n>3$. 

Let us place Theorem  \ref{th:proper} in the context of results in the literature.
It has been known since the 1950's that every open Riemann surface embeds properly 
holomorphically in $\c^3$ \cite{Bis,Nar1,Nar2,Rem}. 
Since a holomorphic embedding is also conformal and harmonic, it follows
that every open Riemann surface carries a  proper  conformal minimal embedding into  $\r^6$.
In a different direction, Greene and Wu showed in 1975 \cite{GW} that every open 
$k$-dimensional Riemannian manifold $M^k$ admits a proper (not necessarily conformal) 
harmonic embedding into $\r^{2k+1}$; hence surfaces $(k=2)$ embed properly harmonically 
into $\r^5$. However, the image of a non-conformal harmonic map is not necessarily a minimal surface,
hence Theorem \ref{th:proper} is a refinement of their result  when $M$ is an {\em orientable} surface. 

The optimal result for immersions was obtained by
Alarc\'on and L\'opez who proved that every open Riemann surface carries a proper conformal 
minimal immersion into $\r^3$ \cite{AL1,AL-associated}.

It is well known that Theorem \ref{th:proper} fails in dimension $n=3$.
Indeed, the existence of a proper conformal minimal embedding $M\hra \r^3$ is a very restrictive 
condition on $M$ and there is a rich literature on this subject; see  the recent surveys 
\cite{MP1,MP2} and the references therein. 
Note however that every Riemann surface (open or closed) admits a smooth proper conformal
(but not necessarily minimal!) embedding into $\r^3$ according to R\"uedy \cite{Ruedy};
see also Garsia \cite{Garsia} for a partial result in this direction.

It remains an open problem whether Theorem \ref{th:proper}  holds in dimension $n=4$:

\begin{problem}\label{pr:properemb1}
Does every open Riemann surface admit a proper conformal minimal embedding into $\r^4$?
\end{problem}

Motivated by the result of Greene and Wu \cite{GW}, 
another interesting but less ambitious open question is whether every open (orientable) Riemannian surface 
admits a harmonic embedding into $\r^4$. These problems seem nontrivial also for nonproper embeddings. 

Since every holomorphic embedding is conformal and harmonic (hence minimal), 
Problem \ref{pr:properemb1} is related to the analogous long standing open problem whether every open 
Riemann surface admits a proper holomorphic  embedding into $\c^2$. 
(See the survey of Bell and Narasimhan \cite[Conjecture 3.7, p.\ 20]{BN}, 
and  \cite[Sections 8.9--8.10]{F2011}.)  For recent progress on this problem, we refer to the articles
of Forstneri\v c and Wold \cite{FW1,FW2} and the references therein. 
In particular, the following result is proved in \cite{FW1}:
{\em Let $M$ be a compact bordered Riemann surface
with nonempty boundary $bM$. If $M$ admits a (nonproper) holomorphic embedding into $\c^2$, 
then its interior $\mathring M=M\setminus bM$ admits a proper holomorphic embedding into $\c^2$.} 
The same problem makes sense for conformal minimal embeddings and it naturally appears 
as a first approach to Problem \ref{pr:properemb1}:

\begin{problem}\label{pr:properemb2}
Let $M$ be a compact bordered Riemann surface with nonempty boundary $bM$.
Assume that $M$ admits a conformal minimal embedding $M\hookrightarrow\r^4$ 
of class $\Cscr^1(M)$. Does the interior $\mathring M$ admit a proper  conformal minimal 
embedding into $\r^4$? 
\end{problem}

It was proved in \cite{FW2}  that every circular domain in
the Riemann sphere $\cp^1$ (possibly infinitely connected)  
admits a proper holomorphic embedding into $\c^2$, hence
a proper conformal minimal embedding into $\r^4$.

The constructions in \cite{FW1,FW2} use the theory of 
holomorphic automorphisms of complex Euclidean spaces.
After exposing and sending to infinity a point in each boundary component of $M$,
a proper holomorphic embedding $\mathring M\hookrightarrow\c^2$ 
is obtained by successively pushing the image of the boundary $bM$ to infinity by holomorphic automorphisms 
of $\c^2$.  Such approach does not seem viable for conformal 
minimal embeddings since this class of maps is only preserved under rigid transformations of $\r^n$. 
It is therefore a challenging problem to find suitable methods that could work for conformal minimal embeddings.

On the other hand, there are no topological obstructions to these questions since Alarc\'on and L\'opez \cite{AL-C2} 
proved that every open orientable surface admits a smooth proper embedding in $\c^2$ whose image is a complex curve. 
For bordered orientable surfaces of finite topology this was shown earlier by \v Cerne and Forstneri\v c \cite{CF}.

The results in this paper, as well as the methods used in their proofs, are influenced by the recent work 
\cite{AF2} of the first two named authors  who proved results analogous to Theorems \ref{th:desing} and \ref{th:proper} 
for a certain class of {\em directed holomorphic immersions} (including null curves) of Riemann surfaces 
into $\c^n$, $n\geq 3$. A {\em null curve} $M \to \c^n$ is a holomorphic immersion whose derivative lies 
in the punctured null quadric $\Agot^*=\Agot\setminus \{0\}\subset \c^n$;
see (\ref{eq:Agot}). The real part of a null curve is a conformal minimal immersion
$M\to \r^n$. (Obviously, the real part of an embedded null curve is not necessarily embedded.) 
These techniques exploit the close connection between minimal surfaces in $\r^n$ and
modern Oka theory in complex analysis. The most relevant point
for the global approximation results (in particular, Theorems \ref{th:desing}, \ref{th:proper} and \ref{th:Mergelyan})
is that the punctured null quadric $\Agot^*$ is an {\em Oka manifold}. 
(The simplest description of this class of complex manifolds can be found in L\'arusson's 
AMS Notices article \cite{Lar}. Roughly speaking, maps $M\to X$ from any Stein manifold $M$ 
(in particular, from any open Riemann surface) to an Oka manifold $X$ satisfy all 
forms of Oka principle. Recent expository articles on this subject are \cite{F-Oka,FL-survey};
a more comprehensive treatment can be found in the monograph \cite{F2011}.)
On the other hand, local results which pertain to properties of conformal minimal
immersions on compact bordered Riemann surfaces, such as Theorem \ref{th:local} 
(the structure theorem) and Theorem \ref{th:desingBRS} (the general position theorem),
do not require that $\Agot^*$ is Oka. This dichotomy has already been pointed out in \cite{AF2}.

The proofs of our main results follow the pattern that has been established in \cite{AF2}.  
We exhaust the Riemann surface $M$ by an increasing sequence of compact, smoothly bounded, 
Runge domains $M_1\subset M_2\subset M_3\subset \cdots$
such that for every $j\in \n$, $M_{j+1}$ deformation retracts either onto $M_j$
(the so called {\em noncritical case}), or onto $M_j\cup C_j$
where $C_j$ is a smooth embedded arc in $M\setminus M_j$ attached with endpoints to $M_j$ 
(the {\em critical case}). We proceed recursively. 
Assume inductively that we have already found a conformal minimal immersion (or embedding) $u_j\colon M_j\to \r^n$
on a neighborhood of $M_j$ for some $j\in\n$. (The initial set $M_1$ is chosen to be a small 
neighborhood of the compact set $K\subset M$ on which we wish to approximate an initially given conformal 
minimal immersion or embedding.)
We embed its derivative $\di u_j$ into a period dominating  holomorphic spray of maps into
the punctured null quadric $\Agot^*$ (see Lemma \ref{lem:deformation}).
Since $\Agot^*$ is an Oka manifold, we can approximate this spray by a spray
defined on a neighborhood of $M_{j+1}$. (In the noncritical case, this is a direct application
of the Oka property. The critical case requires additional work, using generalized conformal minimal immersions 
on admissible sets, see Definitions \ref{def:admissible} and \ref{def:generalized}. When extending $u_j$ to the arc 
$C_j$, we must ensure the correct value of the integral $\int_{C_j} \Re(\di u_j)$.)  
The period domination property ensures that the new spray still contains
an element with vanishing real periods on all closed curves in $M_{j+1}$, and hence
it integrates to a conformal minimal immersion $u_{j+1}\colon M_{j+1}\to \r^n$.
If $n\ge 5$, we can furthermore arrange that $u_{j+1}$ is an embedding by 
using a general position argument (see Theorem \ref{th:desingBRS}), similar to the one obtained 
in \cite[Theorem 2.4]{AF2} for directed holomorphic immersions $M\to \c^n$.  If the approximations
are sufficiently close at every step, then the sequence $(u_j)_{j\in\n}$ converges to a conformal
minimal immersion $u=\lim_{j\to \infty} u_j\to \r^n$ (embedding if $n\ge 5$).  
The proof of Theorem \ref{th:proper} (concerning the existence of {\em proper} conformal minimal embeddings
$M\to\r^n$ for $n\ge 5$) proceeds similarly, but uses  a more precise version 
of Mergelyan's approximation with the control of some of the component functions;
see Lemma \ref{lem:Mergelyan2}. 

A few words concerning the organization of the paper.  In Section  \ref{sec:prelim}, 
we collect the preliminary material and introduce the relevant definitions.
In Section  \ref{sec:bordered}, we prove some basic local properties of the space of 
conformal minimal immersions of bordered Riemann surfaces to $\r^n$; cf.\ Theorem \ref{th:local}.
In Section  \ref{sec:desing}, we prove the special case of Theorem \ref{th:desing}
for bordered Riemann surfaces (cf.\ Theorem \ref{th:desingBRS}).
In Section  \ref{sec:Mergelyan}, we prove a Mergelyan type approximation theorem for conformal 
minimal immersions of open Riemann surfaces into $\r^n$ for any $n\ge 3$; 
see Theorem \ref{th:Mergelyan}. By combining all these results, we then prove
Theorem \ref{th:desing} in Section \ref{sec:th1.1} (cf.\ Theorem \ref{th:desing2})
and Theorem \ref{th:proper} in Section \ref{sec:proper} (cf.\ Theorem \ref{th:proper2}).


\section{Notation and preliminaries} 
\label{sec:prelim}

Let $n\geq 3$ be a natural number, and let $M$ be an open Riemann surface. 
An immersion $u=(u_{1},u_{2},\ldots,u_{n}):M\to \r^n$ is {\em conformal}
(angle preserving) if and only if, in any local holomorphic coordinate $z=x+\imath y$ on $M$, 
the partial derivatives $u_x=(u_{1,x},\ldots,u_{n,x})\in \r^n$ and $u_y=(u_{1,y},\ldots,u_{n,y})\in \r^n$
have the same Euclidean length and are orthogonal:
\begin{equation}\label{eq:conformal} 
	|u_x|=|u_y|>0, \qquad u_x\cdotp u_y=0.
\end{equation}
Equivalently,  $u_x \pm \imath u_y\in \c^n\setminus\{0\}$ are {\em null vectors}, i.e., they 
belong to the {\em null quadric}
\begin{equation}
\label{eq:Agot}
	\Agot=\{z=(z_1,z_2,\ldots,z_n)\in\c^n\colon z_1^2+z_2^2+\cdots + z_n^2=0\}.
\end{equation}

Since $\Agot$ is defined by a homogeneous quadratic holomorphic equation 
and is smooth away from the origin, the punctured null quadric  
$\Agot^*=\Agot\setminus\{0\}=\Agot_{\reg}$ is an {\em Oka manifold} 
(see Example 4.4 in \cite[p.\ 743]{AF2}). This means that:

\begin{remark}\label{rem:Oka}
Maps $M\to \Agot^*$ from any Stein manifold (in particular, from any open Riemann surface) 
satisfy  all forms of the {\em Oka principle} \cite[Theorem 5.4.4]{F2011}. 
\end{remark}

The exterior derivative on $M$ splits into the sum $d=\di+\dibar$ of the $(1,0)$-part $\di$
and the $(0,1)$-part $\dibar$. In any local holomorphic coordinate $z=x+\imath y$ on $M$ we have 
\begin{equation}\label{eq:di-u} 
	2\di u= (u_x - \imath u_y)dz, \quad 2\dibar u= (u_x + \imath u_y)d\bar z.
\end{equation}
Hence (\ref{eq:conformal}) shows that $u$ is conformal  if and only  if  
the differential $\di u=(\di u_1,\ldots,\di u_n)$ satisfies the nullity condition
\begin{equation}\label{eq:sumuzero}
	 (\di u_1)^2 + (\di u_2)^2 + \cdots + (\di u_n)^2 =0.
\end{equation}

Assume now that $M$ is an open Riemann surface and 
$u\colon M\to \r^n$ is a conformal immersion. It is classical (cf.\ Osserman \cite{Osserman}) that 
$\Delta u =2 \mu \mathbf H$, where $\mathbf H\colon M\to\r^n$ is the mean curvature vector of $u$
and $\mu>0$ is a positive function. (In local isothermal coordinates $x+\imath y$ on $M$ 
we have $\mu=||u_x||^2=||u_y||^2$.)
Hence $u$ is minimal ($\mathbf H=0$)  if and only if it is harmonic  ($\Delta u=0$). 
If $v$ is any local harmonic conjugate of $u$ on $M$, then the Cauchy-Riemann equations imply that 
\[
	\di(u+\imath v) = 2\di u = 2\imath\, \di v.
\]
In particular, the differential $\di u$ of any conformal minimal immersion is a $\c^n$-valued 
holomorphic 1-form satisfying (\ref{eq:sumuzero}).

The {\em conjugate differential} of a smooth map $u:M\to\r^n$ is defined by  
\[
	d^c u= \imath(\dibar u - \di u) = 2 \Im (\di u).  
\]
We have that
\[
	2\di u = du + \imath d^c u, \quad 
	dd^c u= 2\imath\,  \di\dibar u = \Delta u\cdotp  dx\wedge dy.
\] 
Thus $u$ is harmonic if and only if $d^c u$ is a closed vector valued $1$-form on $M$, 
and $d^c u=dv$ holds for any local harmonic conjugate $v$ of $u$. 

The {\em flux map}  of a harmonic map $u\colon M \to\r^n$  
is the group homomorphism $\Flux_u\colon H_1(M;\z)\to\r^n$ given by
\begin{equation} \label{eq:flux}
	\Flux_u([C])=\int_C d^c u, \qquad 	[C]\in H_1(M;\z).
\end{equation}
The integral on the right hand side is independent of the choice of path in a given homology class, 
and we shall write $\Flux_u(C)$ for $\Flux_u([C])$ in the sequel.

Fix a nowhere vanishing holomorphic $1$-form $\theta$ on $M$.  
(Such exists by the Oka-Grauert principle, cf.\ Theorem 5.3.1 in \cite[p.\ 190]{F2011}.)
It follows from (\ref{eq:sumuzero}) that 
\begin{equation}\label{eq:di-u2}
	2\di u= f \theta
\end{equation}	
where $f=(f_1,\ldots,f_n)\colon M\to \Agot^*$ is a holomorphic map satisfying 
\[
	\int_C \Re (f\theta)= \int_C du=0\quad \text{for any closed curve $C$ in $M$}.
\]
Conversely, associated to any holomorphic map $f\colon M\to\c^n$ is the {\em period homomorphism} 
$\Pcal(f)\colon H_1(M;\z) \to \c^n$ defined on any closed curve $C\subset M$ by 
\[
	\Pcal(f)(C) = \int_{C} f\theta. 
\]
The map $f$ corresponds to a conformal minimal immersion $u\colon M\to\r^n$ as in (\ref{eq:di-u2})
if and only if $f(M)\subset \Agot^*$ and $\Re(\Pcal(f))=0$; in this case,  
$u(x) =  \int^x \Re(f\theta)$ $(x\in M)$ and
\begin{equation}\label{eq:FP}
	\Flux_u = \Im (\Pcal(f)) \colon H_1(M;\z) \to \r^n. 
\end{equation}

In view of Remark \ref{rem:Oka}, the above discussion connects the theory of minimal surfaces in $\r^n$ 
to the theory of Oka manifolds.
For the latter, see the expository articles \cite{F-Oka,FL-survey} and the monograph \cite{F2011}.

Next, we introduce the mapping spaces that will be used in the paper. 

If $M$ is an open Riemann surface, then $\Oscr(M)$ is the algebra of holomorphic functions  $M\to\c$, 
$\Oscr(M,X)$ is the space of  holomorphic mappings $M\to X$ to a complex manifold $X$, 
and $\CMI(M,\r^n)$  is the set of conformal minimal immersions $M\to \r^n$. 
These spaces are endowed with the compact-open topology.

If $K$ is a compact subset of $M$, then $\Oscr(K)$ denotes 
the set of all holomorphic functions on open neighborhoods of $K$ in $M$
(in the sense of germs on $K$). A compact set $K\subset M$
is said to be {\em $\Oscr(M)$-convex} 'if $K$ equals its holomorphically convex hull 
\[
	\wh K=\{x\in M\colon |f(x)| \le \sup_K |f|\ \ \forall f\in \Oscr(M)\}. 
\]
If $M$ is an open Riemann surface, then by Runge's theorem $K=\wh K$ if and only if $M\setminus K$ 
does not contain any relatively compact connected components, and this holds precisely when 
every function $f\in\Oscr(K)$ is the uniform limit on $K$ of functions in $\Oscr(M)$; 
for this reason, such $K$ is also  called a {\em Runge set} in $M$. 
(See e.g.\ \cite{Hormander} for these classical results.)

Assume now that $M$ is a {\em compact bordered Riemann surface}, i.e., a  compact connected Riemann surface with 
smooth boundary $\emptyset \ne bM \subset M$ and interior $\mathring M=M\setminus bM$. 
Let $g\ge 0$ be the genus of $M$ and $m\ge 1$ the number of its boundary components of $M$. 
The first homology group $H_1(M;\z)$ is then a free abelian group on $l=2g+m-1$ generators
whose basis  is given by smoothly embedded loops $\gamma_1,\ldots,\gamma_l\colon \s^1 \to \mathring M$ that only 
meet at a chosen base point $p\in \mathring M$. (Here, $\s^1$ denotes the circle.) 
Let $C_j=\gamma_j(\s^1)\subset M$ denote the trace of $\gamma_j$. 
The union $C=\bigcup_{j=1}^l C_j$ is then a wedge of $l$ circles with vertex $p$. 

Given $r\in \z_+$,  we denote by $\Ascr^r(M)$ the space of all functions
$M\to \c$ of class $\Cscr^r(M)$ that are holomorphic in $\mathring M$. 
More generally, for any complex manifold $X$ we let $\Ascr^r(M,X)$ denote the space 
of maps $M\to X$ of class $\Cscr^r$ which are holomorphic in $\mathring M$.
We write $\Ascr^0(M)=\Ascr(M)$ and $\Ascr^0(M,X)=\Ascr(M,X)$. 
Note that $\Ascr^r(M,\c^n)$ is a complex Banach space, and for any complex manifold $X$ the space $\Ascr^r(M,X)$ 
is a complex Banach manifold modeled on $\Ascr^r(M,\c^n)$ with $n=\dim X$ (see \cite[Theorem 1.1]{F2007}). 

For any $r\in \n$ we denote by $\CMI^r(M,\r^n)$ the set of all  conformal minimal immersions $M\to\r^n$ 
of class $\Cscr^r(M)$. More precisely, an immersion $F\colon M\to \r^n$ of class $\Cscr^r$ belongs to 
$\CMI^r(M,\r^n)$ if and only if $\di F$ is a $(1,0)$-form of class $\Cscr^{r-1}(M)$ that has
range in the punctured null quadric $\Agot^*$ (\ref{eq:Agot}) and is holomorphic in the interior 
$\mathring M=M\setminus \partial M$. We write $\CMI^1(M,\r^n)=\CMI(M,\r^n)$.

A compact bordered Riemann surface $M$ can be considered as a smoothly bounded compact domain 
in an open Riemann surface $R$. It is classical that each function in $\Ascr^r(M)$ can be approximated 
in the $\Cscr^r(M)$ topology by functions in $\Oscr(M)$. 
The same is true for maps to an arbitrary complex manifold or complex space
(see \cite[Theorem 5.1]{DF}).

The following notions will play an important role in our analysis.

%
%
%
%
\begin{definition}
\label{def:nondegenerate}
Let $M$ be a connected open or bordered Riemann surface, let $\theta$ be a nowhere vanishing holomorphic $1$-form on $M$, 
and let $\Agot$ be the null quadric (\ref{eq:Agot}).
\begin{itemize}
\item A holomorphic map $f:M\to \Agot^*$ is said to be {\em nonflat} if the image 
$f(M)$ is not contained in any complex ray $\c\nu\subset \Agot$ of the null quadric. 
A conformal minimal immersion $u\colon M\to \r^n$ is nonflat if 
the map $f=2\di u/\theta: M\to \Agot^*$ is nonflat, or equivalently,  if the image $u(M)$ 
is not contained in an affine plane. 

\vspace{1mm}

\item A holomorphic map $f\colon M\to \Agot^*$ is {\em nondegenerate} if the image 
$f(M)\subset\Agot^*$ is not contained in any linear complex hyperplane of $\c^n$. 
A conformal minimal immersion $u\colon M\to \r^n$ is nondegenerate 
if the map $f=2\di u/\theta \colon M\to \Agot^*$ is. 

\vspace{1mm}

\item  A conformal minimal immersion $u\colon M\to \r^n$ is {\em full} if the image $u(M)$ is not contained 
in an affine hyperplane. 
\end{itemize}
\end{definition}

For a conformal minimal immersion $M\to \r^3$, nonflat, full, and nondegenerate are equivalent notions. 
However, in dimensions $n>3$  we have
\begin{equation}\label{eq:implications}
	\text{Nondegenerate $\Longrightarrow$ Full $\Longrightarrow$ Nonflat},
\end{equation}
but the converses are not true  (see \cite{Osserman}).

If $M$ is an open Riemann surface, we denote by $\CMI_*(M,\r^n)$
(resp.\ $\CMI_{\rm nf}(M,\r^n)$ the subset of $\CMI(M,\r^n)$ consisting of all immersions which are nondegenerate 
(resp.\ nonflat) on every connected component of $M$.
By \eqref{eq:implications}, we have 
\[
	\CMI_*(M,\r^n)\subset \CMI_{\rm nf}(M,\r^n).
\]
The analogous notation $\CMI_*^r(M,\r^n)\subset \CMI_{\rm nf}^r(M,\r^n) \subset \CMI^r(M,\r^n)$
is used for a compact bordered Riemann surface $M$ and $r\in \n$.
Note that $\CMI_*^r(M,\r^n)$ and $\CMI_{\rm nf}^r(M,\r^n)$ are open subsets
of  $\CMI^r(M,\r^n)$.

Since the tangent space $T_z\Agot$ is the kernel at $z$ of the $(1,0)$-form $\sum_{j=1}^n z_j\, dz_j$, 
we have $T_z\Agot = T_w\Agot$ for $z,w\in\c^n\setminus\{0\}$ if and only if $z$ and $w$ are colinear.
This implies

\begin{lemma}\label{lem:nonflat}
A holomorphic map $f\colon M\to \Agot^*$ is nonflat if and only if the linear span of the tangent spaces $T_{f(x)} \Agot$ 
over all points $x\in M$ equals $\c^n$.
\end{lemma}

\begin{remark}\label{ref:flat-nondeg}
The second condition in Lemma \ref{lem:nonflat} was used as the definition of a nondegenerate map in 
\cite[Definition 2.2, p.\ 736]{AF2}. This property enables the construction of period dominating sprays of holomorphic maps
$M\to\Agot^*$ with the given core map $f$ (cf.\ Lemma \ref{lem:deformation}).  
Here, we revert to the standard terminology used in 
minimal surface theory as given by Definition \ref{def:nondegenerate}.
\end{remark}

%
%
%
%

\section{Conformal minimal immersions of bordered Riemann surfaces}\label{sec:bordered}

The following result gives some basic local properties of the space of conformal minimal
immersions of a bordered Riemann surface to $\r^n$. It is analogous to Theorem 2.3 in
\cite{AF2} where similar properties were proved for certain classes of directed 
holomorphic immersions; in particular, for null holomorphic immersions $M\to\c^n$ 
for any $n\ge 3$.

%
%
%
%
\begin{theorem} 
\label{th:local}
Let $M$ be a compact bordered Riemann surface with nonempty boundary $bM$,
and let $n\ge 3$ and $r\ge 1$ be integers.  
\begin{itemize}
\item[\rm(a)]
	Every conformal minimal immersion $u\in \CMI^r(M,\r^n)$ can be approximated in the 
	$\Cscr^r(M)$ topology by nondegenerate conformal minimal immersions 
	$\tilde u\in \CMI^r_*(M,\r^n)$ satisfying $\Flux_u=\Flux_{\tilde u}$.
\vspace{1mm}
\item[\rm (b)] 
	Each of the spaces $\CMI_*^r(M,\r^n)$ and $\CMI_{\rm nf}^r(M,\r^n)$
	is a real analytic Banach manifold with the natural $\Cscr^r(M)$ topology.
\vspace{1mm}
\item[\rm (c)] 
	If $M$ is a smoothly bounded compact domain in a Riemann surface $R$, then every 
	$u\in \CMI^r(M,\r^n)$ can be appproximated in the $\Cscr^r(M)$
	 topology by conformal minimal immersions defined on open neighborhoods of $M$ in $R$.
\end{itemize}
\end{theorem}

\begin{proof}
Fix a nowhere vanishing holomorphic $1$-form $\theta$ on $M$.
Choose a basis $\{\gamma_j\}_{j=1}^l$ of the homology group $H_1(M;\z)$ and denote by 
\[
	P=(P_1,\ldots, P_l)\colon \Ascr(M,\c^n)\to (\c^n)^l
\]
the period map whose $j$-th component, applied to $f\in \Ascr(M,\c^n)$, equals 
\begin{equation}
\label{eq:P}
	P_j(f) = \int_{\gamma_j} f\theta = 
	\int_0^1 f(\gamma_j(t))\, \theta(\gamma_j(t),\dot{\gamma_j}(t))\, dt \in \c^n.
\end{equation}

%
%
\noindent {\em Proof of (a).} 
For simplicity of notation we assume that $r=1$; the same proof applies for any $r\in\n$.

Let $u\colon M\to \r^n$ be a degenerate conformal minimal immersion.
The map $f=2\di u/\theta \colon M\to \Agot^*$ is continuous on $M$ and 
holomorphic in $\mathring M$, and the linear span of $f(M)$ is a $k$-dimensional linear complex subspace 
$\Pi\subset \Agot$, $1\leq k<n$. 
Fix distinct points $\{p_1,\ldots,p_k,q_1,\ldots,q_{n-k}\}\in M$ such that $\{f(p_1),\ldots,f(p_k)\}$ is a basis of $\Pi$.  
Choose a nonconstant  function $h\in \Ascr(M)$ such that $h(p_i)=0$ for all $i=1,\ldots,k$, and $h(q_j)=1$ 
for all $j=1,\ldots,n-k$.  Choose a holomorphic vector field $V$ on $\c^n$ tangential to $\Agot$
such that $\{f(p_1),\ldots,f(p_k),V(f(q_1)),\ldots,V(f(q_{n-k}))\}$ is a basis of $\c^n$.
Let $t\mapsto \phi(t,z)$ denote the flow of $V$ for small complex values of time $t$, with $\phi(0,z)=z$.
For any $g\in \Ascr(M)$ near the zero function we define the map 
$\Phi(g)\in \Ascr(M,\Agot^*)$ by 
\[
	\Phi(g)(x)= \phi(g(x)h(x), f(x)),\quad x\in M.
\]
Clearly $\Phi(0)=f$. Consider the holomorphic map
\[
	\Ascr(M) \ni g \longmapsto P(\Phi(g)) \in (\c^n)^l.
\]
Since the space $\Ascr(M)$ is infinite dimensional, 
there is a function $g\in \Ascr(M)\setminus \{0\}$ arbitrarily close to
the zero function such that $P(\Phi(g))=P(\Phi(0))= P(f)$; in particular, $\Re P(\Phi(g))=0$.
For such $g$,  the map $\tilde f = \Phi(g) \colon M\to \Agot^*$ 
integrates to a conformal minimal immersion 
$\tilde u(x) = u(p)+\int_p^x \Re(\tilde f\theta)$ that is close to $u$ and satisfies $\Flux_u=\Flux_{\tilde u}$.
For a generic choice of points $q_j'\in M$ near $q_j$, $j=1,\ldots,n-k$, we have that $g(q_j')h(q_j')\ne 0$ 
and $\{f(p_1),\ldots,f(p_k),V(f(q_1')),\ldots,V(f(q_{n-k}'))\}$ is a basis of $\c^n$.
Hence, $\{f(p_1),\ldots,f(p_k),\tilde f(q_1'),\ldots,\tilde f(q_{n-k}')\}$ is a basis of $\c^n$
provided that $g(q_j')\in \c^*$ is close enough to $0$ for every $j=1,\ldots,n-k$. 
Since $g(p_i)h(p_i)=0$, we also have that $\tilde f(p_i)= f(p_i)$ for $i=1,\ldots,k$, and 
$\tilde f$ is nondegenerate.  This proves part (a). 

In the proof of part (b), we shall need the following version of  \cite[Lemma 5.1]{AF2}. 

%
%
\begin{lemma} 
\label{lem:deformation}
Let $r\in\z_+$, and let $f\in \Ascr^r(M,\Agot^*)$ be a nonflat map.
There exist an open neighborhood $U$ of the origin in the Euclidean space 
$\c^N$ for some $N\in\n$ and a map 
\[
	U\times M\ni (\zeta,x) \longmapsto \Phi_f(\zeta,x)\in \Agot^*
\] 
of class $\Ascr^r(U\times M,\Agot^*)$ such that $\Phi_f(0,\cdotp)=f$ 
and the period map $U\ni \zeta\mapsto P(\Phi_f(\zeta,\cdotp)) \in (\c^n)^l$ (\ref{eq:P})
is submersive at $\zeta=0$.  Furthermore, there is a neighborhood $V$ of $f$ in 
$\Ascr^r(M,\Agot^*)$ such that the map $V\ni h \mapsto \Phi_{h}$ can be chosen
holomorphic in $h$.
\end{lemma}

The cited result  \cite[Lemma 5.1]{AF2} is stated for the case when $P(f)=0$, $r=0$, and $f\colon M\to \Agot^*$ is 
nondegenerate as opposed to nonflat, the latter being a weaker condition when $n>3$; see \eqref{eq:implications}.
(In fact, \cite[Lemma 5.1]{AF2} applies to more general conical subvarieties of $\c^n$.)
However, the proof given there applies in the present situation since 
it only uses the condition that the linear span of the tangent spaces $T_{f(x)} \Agot$ over all 
$x\in M$ equals $\c^n$. By Lemma \ref{lem:nonflat}, this holds if and only if the map 
$f\colon M\to \Agot^*$ is nonflat. See also Remark \ref{ref:flat-nondeg}.

A holomorphic family of maps $\Phi_f(\zeta,\cdotp)\colon M\to \Agot^*$ $(\zeta\in U\subset\c^N)$
in Lemma \ref{lem:deformation} is called a {\em period dominating spray
with core $\Phi_f(0,\cdotp) = f$ and with values in $\Agot^*$}.
In \cite[Proof of Lemma 5.1]{AF2} it was shown that there is a spray with these properties
given by 
\begin{equation} \label{eq:Phi}
	\Phi_f(\zeta,x) = \Phi(\zeta,x,f(x)) \in \Agot^*,
\end{equation}
where $\Phi\colon U\times M\times \Agot\to \Agot$ is a holomorphic map of the form
\begin{equation}\label{eq:Phi2}
	\Phi(\zeta,x,z) = \phi^1_{\zeta_1g_1(x)} \circ \cdots  \circ \phi^N_{\zeta_N g_N(x)} (z) \in \Agot,
\end{equation}
where $z\in \Agot$, $x \in M$, $\zeta=(\zeta_1,\ldots,\zeta_N)\in U\subset \c^N$,
$g_1,\ldots,g_N$ are holomorphic functions on $M$, and $\phi^j_t$
is the flow of a holomorphic vector field $V_j$ on $\c^n$ that is tangential to $\Agot$.

%
%
\medskip \noindent {\em Proof of (b).} 
By \cite[Theorem 1.1]{F2007}, the space $\Ascr^{r-1}(M,\Agot^*)$ is a 
complex Banach manifold modeled on the complex Banach space 
$\Ascr^{r-1}(M,\c^{n-1})$ (since $\dim \Agot^*=n-1$). Set 
\[
	\Ascr_0^{r-1}(M,\Agot^*)=\{f\in \Ascr^{r-1}(M,\Agot^*): \Re(P(f))=0\},
\]
where $P\colon \Ascr^{r-1}(M,\c^n)\to (\c^n)^l$ is the (holomorphic) period map (\ref{eq:P}). 
Let $\Ascr_{0,*}^{r-1}(M,\Agot^*)$ denote the open subset of $\Ascr_0^{r-1}(M,\Agot^*)$ consisting of all
nondegenerate maps (see Definition \ref{def:nondegenerate}). Since nondegenerate maps are nonflat,
Lemma \ref{lem:deformation} implies that the differential $dP_{f_0}$ of the restricted 
period map $P\colon \Ascr^{r-1}(M,\Agot^*)\to (\c^n)^l$ at any point 
$f_0\in \Ascr_{0,*}^{r-1}(M,\Agot^*)$ has maximal rank equal to $ln$. 
By the implicit function theorem, it follows that $f_0$ admits an open neighborhood 
$\Omega \subset \Ascr^{r-1}(M,\Agot^*)$ such 
$\Omega\cap \Ascr_0^{r-1}(M,\Agot^*)=\Omega \cap \Ascr_{0,*}^{r-1}(M,\Agot^*)$ 
is a real analytic Banach submanifold of $\Omega$ which is parametrized by the kernel of the 
real part $\Re(dP_{f_0})$ of the differential of $P$ at $f_0$;
this is a real codimension $ln$ subspace of the complex Banach space $\Ascr^{r-1}(M,\c^{n-1})$
(the tangent space of the complex Banach manifold $\Ascr^{r-1}(M,\Agot^*)$).
This shows that $\Ascr_{0,*}^{r-1}(M,\Agot^*)$ is a real analytic Banach manifold. 
The integration $x\mapsto v+\int_p^x \Re(f\theta)$ $(x\in M)$, with an arbitrary choice 
of the initial value $v\in \r^n$ at a chosen base point $p\in M$, provides an isomorphism between the 
Banach manifold  $\Ascr_{0,*}^{r-1}(M,\Agot^*) \times \r^n$  and the space $\CMI^r_*(M,\r^n)$, 
so the latter is also a Banach manifold. The same argument applies to the space $\CMI^r_{\rm nf}(M,\r^n)$
of nonflat maps. This completes the proof of part (b).  

%
%
\medskip \noindent {\em Proof of (c).} 
Without loss of generality, we may assume that $M$ is connected.
Let $u\in \CMI^r(M,\r^n)$. By part (a) we may assume that $u$ is nondegenerate. 
Write $2\di u=f\theta$, where $f\colon M\to \Agot^*$ is a nondegenerate holomorphic map. 
Let $\Phi_f$ be a period dominating spray of 
conformal minimal immersions with the core $f$, furnished by Lemma \ref{lem:deformation}. 
By \cite[Theorem 1.2]{DF2008}, we can approximate $f$ uniformly on $M$ by holomorphic maps 
$\tilde  f\colon V\to \Agot^*$ defined on an open neighborhood $V$ of $M$ 
in $R$. The associated spray $\Phi_{\tilde f}$ is then defined and 
holomorphic on a neighborhood $\wt U\times \wt V \subset \c^N\times R$ 
of $\{0\}\times M$. If $\tilde f$ is sufficiently uniformly close to $f$ on $M$, 
then the domain and the range of the period map $P(\Phi_{\tilde f})$   
are so close to those of $P(\Phi_f)$ that the range of $P(\Phi_{\tilde f})$ 
contains the point $P(f)\in \c^{ln}$. (Note that the components of $P(f)$ are purely imaginary
since $f$ corresponds to a conformal minimal immersion.)
Hence $\tilde f$ can be approximated in $\Cscr^{r-1}(M)$ 
by a holomorphic map $h \in \Oscr(W,\Agot^*)$ on a connected open 
neighborhood $W\subset R$ of $M$ satisfying $P(h)=P(f)$; in particular, 
$\Re(P(h))=0$. The integral $\tilde u(x)=u(p)+\int_p^x \Re(h\, \theta)$ is then 
a conformal minimal immersion in a neighborhood of $M$ in $R$
which approximates $u$ in $\Cscr^r(M,\r^n)$.
\end{proof}

%
\section{Desingularizing conformal minimal immersions}
\label{sec:desing}

In this section we prove the following general position theorem (a special case  of 
Theorem \ref{th:desing}) for conformal minimal immersions of bordered Riemann surfaces.

%
%
%
%
\begin{theorem}
\label{th:desingBRS}
Let $M$ be a compact bordered Riemann surface and let $n\ge 5$ and $r\ge 1$ be integers.
Every conformal minimal immersion $u\in  \CMI^r(M,\r^n)$  can be approximated
arbitrarily closely in the $\Cscr^r(M)$ topology by a conformal minimal embedding 
$\tilde u \in  \CMI^r(M,\r^n)$ satisfying $\Flux_{\tilde u}=\Flux_u$.
\end{theorem}

Since the set of embeddings $M\to\r^n$ is clearly open in the set of immersions of class 
$\Cscr^r(M)$ for any $r\ge 1$ and $\CMI^r(M,\r^n)$ is a closed subset of the Banach 
space $\Cscr^r(M,\r^n)$ (hence a Baire space), we immediately get

\begin{corollary}
Let $M$ be a compact bordered Riemann surface. For every pair of integers
$n\ge 5$ and $r\ge 1$  the set of conformal minimal embeddings $M\hookrightarrow \r^n$ of class $\Cscr^r(M)$ is
residual (of the second category) in the Baire space $\CMI^r(M,\r^n)$.
\end{corollary}

\begin{proof}[Proof of Theorem \ref{th:desingBRS}]
In view of Theorem \ref{th:local} (parts (a) and (c)), we may assume that $M$ is a smoothly 
bounded domain in an open Riemann surface $R$ and $u$ is a nondegenerate 
(see Definition \ref{def:nondegenerate}) conformal minimal immersion in
an open neighborhood of $M$ in $R$. 

We associate to $u$ the {\em difference map} $\delta u\colon M\times M\to \r^n$ defined by
\[
	\delta u(x,y)=u(y)-u(x), \qquad x,y\in M.
\]
Clearly, $u$ is injective if and only if $(\delta u)^{-1}(0)= D_M:=\{(x,x): x\in M\}$. 
Since $u$ is an immersion, it is locally injective, and hence there is an open neighborhood 
$U\subset M\times M$ of the diagonal  $D_M$ such that $\delta u$ does not assume the value 
$0\in \r^n$ on $\overline U\setminus D_M$. To prove the theorem, it suffices to find arbitrarily close 
to $u$ another conformal minimal immersion $\tilde u \colon M\to\r^n$ whose difference map 
$\delta \tilde u$  is transverse to the origin $0\in \r^n$ on $M\times M\setminus U$.
Since $\dim_\r M\times M=4<n$, this will imply that $\delta\tilde u $ does not assume the 
value zero on $M\times M\setminus U$, so $\tilde u (x)\ne \tilde u (y)$ if $(x,y)\in M\times M\setminus U$. 
If $(x,y)\in U \setminus D_M$ then $\tilde u (x)\ne \tilde u (y)$ provided that 
$\tilde u$ is close enough to $u$, so $\tilde u $ is an embedding. 

To find such $\tilde u$, we shall  construct a neighborhood $\Omega \subset \r^N$ 
of the origin in a Euclidean space and a real analytic map $H\colon \Omega \times M \to \r^n$ 
satisfying the following properties:
\begin{itemize}
\item[\rm (a)] $H(0,\cdotp)=u$,
\item[\rm (b)] $H(\xi,\cdotp)\colon M\to \r^n$ is a conformal minimal immersion of class
$\Cscr^r(M)$ for every $\xi \in \Omega$, and 
\item[\rm (c)]  the difference map $\delta H\colon \Omega \times M\times M \to \r^n$, defined by 
\begin{equation}
\label{eq:difference}
	\delta H(\xi,x,y) = H(\xi,y)-H(\xi,x), \qquad \xi\in \Omega, \ \ x,y\in M, 
\end{equation}
is a submersive family  on $ M\times M\setminus U$, in the sense that the partial differential 
\begin{equation} \label{eq:pd}
	d_\xi|_{\xi=0} \, \delta H(\xi,x,y) \colon \r^N \to \r^n
\end{equation}
is surjective for every $(x,y)\in M\times M\setminus U$. 
\end{itemize}

Assume for a moment that such $H$ exists.
By compactness of $M\times M \setminus U$, the partial differential $d_\xi (\delta H)$ 
(\ref{eq:pd}) is surjective for all $\xi$ in a neighborhood 
$\Omega'\subset \Omega$ of the origin in $\r^N$. 
Hence the map $\delta H \colon M\times M\setminus U\to\r^n$ 
is transverse to any submanifold of $\r^n$, in particular, to the origin $0\in \r^n$. 
The standard transversality argument due to Abraham \cite{Abraham}
(a reduction to Sard's theorem; see also \cite[Section  7.8]{F2011}) 
shows that for a generic choice of $\xi\in\Omega'$,  the difference map 
$\delta H(\xi,\cdotp,\cdotp)$ is transverse to $0\in\r^n$ on $M\times M\setminus U$,
and hence it omits the value $0$ by dimension reasons. 
By choosing $\xi$ sufficiently close to $0\in\r^N$ we thus obtain a 
conformal minimal embedding $\tilde u=H(\xi,\cdotp)\colon M \to \r^n$ close to $u$,
thereby proving the theorem.  

We construct a spray $H$ of conformal minimal immersions, satisfying properties (a)--(c)
above, by  suitably modifying the proof of the corresponding result for directed holomorphic immersions
in \cite[Theorem 2.4]{AF2}.

Fix a nowhere vanishing holomorphic $1$-form $\theta$ on $M$ and write 
$2\di u=f\theta$, where $f\colon M\to \Agot^*$ is a nondegenerate 
holomorphic map (see Definition \ref{def:nondegenerate}). 
The main step in the construction of the spray $H$ is furnished by the following lemma.

%
%
%
%
\begin{lemma}
\label{lem:pq}
{\rm (Assumptions as above.)} 
For every $(p,q)\in M\times M\setminus D_M$ there exists a spray 
$H=H^{(p,q)}(\xi,\cdotp) \colon M\to \r^n$ of conformal minimal immersions of class $\Cscr^r(M)$, 
depending analytically on the parameter $\xi$ in a neighborhood of the origin in $\r^n$, 
satisfying properties (a) and (b) above, but with (c) replaced by the following property:
\begin{itemize}
\item[\rm (c')]   the partial differential $d_\xi|_{\xi=0} \, \delta H(\xi,p,q) \colon \r^n \to \r^n$ 
is an isomorphism. 
\end{itemize}
\end{lemma}

\begin{proof}
Let $\Lambda \subset M$ be a smooth embedded arc connecting $p$ to $q$. Pick a point $p_0\in M\setminus\Lambda$ 
and closed loops $C_1,\ldots, C_l \subset M\setminus\Lambda$ based at $p_0$ and forming a Runge basis of $H_1(M;\z)$. 
Set $C=\bigcup_{j=1}^l C_j$. Let $\gamma_j\colon [0,1]\to C_j$ ($j=1,\ldots,l$) and $\lambda\colon [0,1]\to \Lambda$ be 
smooth parametrizations of the respective curves.  

Since $u$ is nonflat, Lemma \ref{lem:nonflat} and the Cartan extension theorem
furnish holomorphic vector fields $V_1,\ldots,V_n$ on $\c^n$, tangential to $\Agot$, 
and points $x_1,\ldots, x_n\in \Lambda\setminus \{p,q\}$ such that, 
setting $z_i=f(x_i)\in \Agot^*$, the vectors $V_i(z_i)$ for $i=1,\ldots,n$ span $\c^n$. 
Let $t_i\in (0,1)$ be such that $\lambda(t_i)=x_i$. Let $\phi^i_t$ denote the flow of $V_i$. 
Choose smooth functions $h_i\colon C\cup \Lambda \to \r_+$ $(i=1,\ldots,n)$ that vanish at the endpoints 
$p,q$ of $\Lambda$ and on the curves $C$; their values on $\Lambda$ will be chosen later. 
Let $\zeta=(\zeta_1,\ldots,\zeta_n) \in \c^n$. Consider the map 
\[
	\psi_f(\zeta,x)=\phi^1_{\zeta_{1}h_{1}(x)} \circ \cdots  \circ \phi^n_{\zeta_{n}h_{n}(x)} (f(x)) 
	\in \Agot^*,\quad x\in C\cup\Lambda.
\] 
Clearly it is holomorphic in the variable $\zeta\in\c^n$ near the origin,
$\psi_f(0,\cdotp)=f$, and $\psi_f(\zeta,x)=f(x)$ if $x\in C\cup \{p,q\}$ (since $h_i=0$ on $C\cup \{p,q\}$). We have that
\[
	\frac{\di \psi_f(\zeta,x)}{\di \zeta_{i}}\bigg|_{\zeta=0} = h_{i}(x)\, V_i(f(x)), \quad x\in C\cup \Lambda, \quad i=1,\ldots,n.
\]
By choosing the function $h_i$ to have support concentrated near the point 
$x_i=\lambda(t_i) \in \Lambda$, we can arrange that for every $i=1,\ldots,n$ we have that
\[
	\int_0^1 h_{i}(\lambda(t))\, V_i(f(\lambda(t))) \, \theta(\lambda(t),\dot\lambda(t)) \,dt
	\approx V_i(z_i) \, \theta(\lambda(t_i),\dot{\lambda}(t_i)) \in \c^n.
\]
Assuming as we may that the above approximations are close enough, the vectors on the left hand side 
of the above display form a basis of $\c^n$. 

Fix a number $\epsilon>0$; its precise value will be chosen later. 
We apply Mergelyan's theorem to find holomorphic functions $g_i\in \Oscr(M)$ such that 
\[
	\sup_{C\cup\Lambda} |g_i-h_i| <\epsilon \quad \text{for}\ \ i=1,\ldots,n. 
\]
Consider the holomorphic maps
\begin{eqnarray}
\label{eq:Psi2}
	\Psi(\zeta,x,z) &=&
	\phi^1_{\zeta_{1}g_{1}(x)} \circ \cdots  \circ \phi^n_{\zeta_{n}g_{n}(x)} (z) \in \Agot,  \cr
	\Psi_f(\zeta,x) &=& \Psi(\zeta,x,f(x)) \in \Agot,
\end{eqnarray}
where $x\in M$, $z\in \Agot$, and $\zeta$ is near the origin in $\c^n$. Note that $\Psi_f(0,\cdotp)=f$.
If the approximations of $h_i$ by $g_i$ are close enough, then the vectors
\begin{equation}\label{eq:derivatives}
	\frac{\di}{\di \zeta_{i}}\bigg|_{\zeta=0} \int_0^1 \Psi_f(\zeta,\lambda(t)) 
	\,\theta(\lambda(t),\dot\lambda(t)) \,dt   
	= \int_0^1 g_{i}(\lambda(t))\, V_i(f(\lambda(t))) \, \theta(\lambda(t),\dot\lambda(t)) \,dt 
\end{equation}
in $\c^n$ are so close to the corresponding vectors $V_i(z_i) \, \theta(\lambda(t_i),\dot{\lambda}(t_i))$ 
$(i=1,\ldots,n)$ that they are $\c$-linearly independent. 

The $\c^n$-valued $1$-form $\Psi_f(\zeta,\cdotp) \,\theta$ 
need not have exact real part, so it may not correspond to the differential of a conformal
minimal immersion. We shall now correct this. 

From the Taylor expansion of the flow of a vector field we see that
\[
	\Psi_f(\zeta,x)=f(x)+\sum_{i=1}^n \zeta_i g_i(x) V_i(f(x)) + o(|\zeta|).
\]
Since $|g_i|<\epsilon$ on $C$, the periods over the loops $C_j$ can be estimated by 
\begin{equation}
\label{eq:estimate-periods}
	\left| \int_{C_j} \big(\Psi_f(\zeta,\cdotp)-f\big)\, \theta \right| \le \eta_0\epsilon |\zeta|
\end{equation}
for some constant $\eta_0>0$ and for all sufficiently small $|\zeta|$. 

Lemma \ref{lem:deformation} gives holomorphic maps $\Phi(\wt \zeta,x,z)$ and $\Phi_f(\wt\zeta,x)=\Phi(\wt\zeta,x,f(x))$ 
(see \eqref{eq:Phi} and \eqref{eq:Phi2}), 
with the parameter $\wt \zeta$ near $0\in\c^{\wt N}$ for some $\wt N\in\n$ and $x\in M$, 
such that $\Phi(0,x,z)=z$ and the differential of the associated period map $\wt \zeta \mapsto P(\Phi_f(\wt\zeta,\cdotp)) \in \c^{ln}$ 
(see (\ref{eq:P})) at the point $\wt\zeta=0$ has maximal rank equal to $ln$. The same is true if the map 
$f\in \Ascr(M,\Agot^*)$ varies locally near the given initial map. 
In particular, we can replace $f$ by the spray $\Psi_f(\zeta,\cdotp)$ and consider the composed map 
\[
	\c^{\wt N}\times \c^n\times M  \ni (\wt \zeta,\zeta,x) \longmapsto \Phi(\wt\zeta,x,\Psi_f(\zeta,x)) 
	\in \Agot^*
\]
which is defined and holomorphic for $(\wt \zeta,\zeta)$ near the origin in $\c^{\wt N}\times \c^n$ and for $x\in M$. 
The implicit function theorem furnishes a $\c^{\wt N}$-valued holomorphic map 
$\wt \zeta=\rho(\zeta)$ near $\zeta=0\in\c^n$, with $\rho(0)=0\in\c^{\wt N}$, 
such that the $\c^n$-valued holomorphic 1-form on $M$ given by
\[
	\Theta_f(\zeta,x,v)= \Phi(\rho(\zeta),x,\Psi_f(\zeta,x))\, \theta(x,v), \quad x\in M,\ v\in T_x M 
\]
satisfies the conditions
\[
	\int_{C_j} \Theta_f(\zeta,\cdotp,\cdotp) = \int_{C_j} f\theta,\qquad j=1,\ldots, l
\]
for every $\zeta\in\c^n$ near the origin.   In particular, the real parts of these periods vanish. 
(The map $\rho=(\rho_1,\ldots,\rho_n)$ also depends on $f$, but we suppressed 
this dependence in our notation.) It follows that the integral
\begin{equation}
\label{eq:H}
	H_u(\zeta,x) =  u(p_0)+ \int_{p_0}^x \Re(\Theta_f(\zeta,\cdotp,\cdotp))  =
	u(p_0) + \int_0^1 \Re(\Theta_f(\zeta,\gamma(t),\dot\gamma(t)))\, dt
\end{equation}
is independent of the choice of the path from the initial point $p_0\in M$ to 
the variable point $x\in M$. Clearly $H_u$ is analytic, $H_u(0,\cdotp)=u$, 
$H_u(\zeta,\cdotp) \colon M\to\r^n$ is a conformal minimal immersion for every 
$\zeta\in\c^n$ sufficiently close to $0$, and the flux homomorphism of $H_u(\zeta,\cdotp) $
equals that of $u$ for every fixed $\zeta$. 
Furthermore, in view of  (\ref{eq:estimate-periods}) we have the estimate
\begin{equation}
\label{eq:est-rho}
	|\rho(\zeta)| \le \eta_1\epsilon |\zeta|
\end{equation}
for some constant $\eta_1>0$ independent of $\epsilon$ and $\zeta$. 

The map $\Phi(\wt \zeta,x,z)$, furnished by Lemma \ref{lem:deformation},  is obtained by composing 
flows of certain holomorphic vector fields $W_j$ on $\Agot$ for the complex times 
$\wt \zeta_j \wt g_j(x)$, where $\wt g_j\in\Oscr(M)$ and $\wt \zeta_j \in\c$. (See (\ref{eq:Phi2}).)
The Taylor expansion of the flow, together with the estimate (\ref{eq:est-rho}), gives
\[	\left| \Phi(\rho(\zeta),x,\Psi_f(\zeta,x)) - \Psi_f(\zeta,x) \right| 
	=  \left| \sum \rho_j(\zeta) \wt g_j(x)   W_j(\Psi_f(\zeta,x)) + o(|\zeta|) \right|
	\le \eta_2\epsilon |\zeta| 
\]      
for some constant $\eta_2>0$ and for all $x\in M$ and all $\zeta$ near the origin in $\c^n$. 
Applying this estimate on the curve $\Lambda$ with the endpoints $p$ and $q$ we get that
\[
	\left| \int_0^1 \Theta_f(\zeta,\lambda(t),\dot\lambda(t))\, dt  - 
	\int_0^1 \Psi_f(\zeta,\lambda(t)) \,\theta(\lambda(t),\dot\lambda(t)) \,dt \right| 
	\le \eta_3 \epsilon |\zeta|
\]
for some constant $\eta_3>0$  independent of $\epsilon$ and $\zeta$.  
If $\epsilon>0$ is chosen small enough, then it follows that the derivatives
\[
	\frac{\di}{\di \zeta_{i}}\bigg|_{\zeta=0} \int_0^1 \Theta_f(\zeta,\lambda(t),\dot\lambda(t)) \, dt \in\c^n,
	\qquad i=1,\ldots,n,   
\]
are so close to the respective vectors in (\ref{eq:derivatives}) 
that they are $\c$-linearly independent.  This means that the holomorphic map 
$\zeta\mapsto \int_0^1 \Theta_f(\zeta,\lambda(t),\dot\lambda(t))\, dt \in \c^n$
is locally biholomorphic near $\zeta=0$. By (\ref{eq:H}), its real part equals
\[
	\int_0^1 \Re(\Theta_f(\zeta,\lambda(t),\dot\lambda(t))) \, dt
	= H_u(\zeta,q)-H_u(\zeta,p)=\delta H_u(\zeta,p,q).
\]
After a suitable $\c$-linear change of coordinates $\zeta=\xi+\imath\eta$ on $\c^n$
it follows that the partial differential 
$\frac{\di}{\di\xi}\big|_{\xi=0} \delta H_u(\xi,p,q) \colon \r^n\to\r^n$
is an isomorphism. The spray 
\begin{equation}\label{eq:Hpq}
H^{(p,q)}(\xi,\cdotp):=H_u(\xi,\cdotp)
\end{equation}
satisfies the conclusion of Lemma \ref{lem:pq}.
\end{proof}

The spray $H_u$ \eqref{eq:Hpq}, furnished by Lemma \ref{lem:pq}, depends
real analytically on $u\in \CMI_*^r(M,\r^n)$ in a neighborhood of a given
nondegenerate conformal minimal immersion $u_0\in \CMI_*^r(M,\r^n)$. 
In particular, if $u(\eta,\cdotp)\colon M\to \r^n$ is a family of conformal minimal immersions 
depending analytically on a parameter $\eta$, then $H_{u(\eta,\cdotp)}(\xi,\cdotp)$ 
depends analytically on $(\xi,\eta)$. This allows us to compose any finite number of such 
sprays just as was done in \cite{AF2}. We recall this operation for two sprays. Suppose that 
$H=H_u(\xi,\cdotp)$ and $G=G_u(\eta,\cdotp)$ are  sprays with 
$H_u(0,\cdotp)=G_u(0,\cdotp)=u$. The composed spray is defined by 
\[
	(H \sharp G)_u(\xi,\eta,x)=G_{H_u(\xi,\cdotp)} (\eta,x),\quad x\in M.
\]
Clearly we have $(H\sharp G)_u(0,\eta,\cdotp)=G_u(\eta,\cdotp)$ and
$(H\sharp G)_u(\xi,0,\cdotp)	=H_u(\xi,\cdotp)$. 
The operation $\sharp$ extends by induction to finitely many factors and is associative. 
(This is similar to the composition of sprays introduced by Gromov \cite{Gromov1989}; 
see also \cite[p.\ 246]{F2011}.)

Pick an  open neighborhood $U\subset M\times M$ of the diagonal $D_M$ 
such that $\overline U\cap (\delta u)^{-1}(0)=D_M$. 
Lemma \ref{lem:pq}  furnishes a finite open covering $\Ucal=\{U_i\}_{i=1}^m$ 
of the compact set $M \times M\setminus U$ and sprays
of conformal minimal immersions $H^i=H^i(\xi^i,\cdotp)\colon M\to \r^n$, 
with  $H^i(0,\cdotp)=u$, where $\xi^i=(\xi^i_1,\ldots,\xi^i_{k_i}) \in \Omega_i\subset \r^{k_i}$, 
such that the difference map $\delta H^i(\xi^i,p,q)$ is submersive at $\xi^i=0$ 
for all $(p,q)\in U_i$. By taking $\xi=(\xi^1,\ldots,\xi^m)\in \r^N$, 
with $N=\sum_{i=1}^m k_i$, and setting
\[
	H(\xi,x) = (H^1 \sharp H^2 \sharp \cdots \sharp H^m)(\xi^1,\ldots,\xi^m,x)
\]
we obtain a spray satisfying properties (a) and (b) whose 
difference map $\delta H$ is submersive on $M\times M\setminus U$ for all $\xi\in \r^N$ 
sufficiently near the origin.  As explained earlier, a generic member $H(\xi,\cdotp)$ 
of this spray (for $\xi$ sufficiently close to $0\in\r^N$) is a conformal minimal embedding 
$M\hra\r^n$. 
\end{proof}


%

\section{ Mergelyan's theorem for conformal minimal immersions to $\r^n$}
\label{sec:Mergelyan}

In this section, we prove a Mergelyan type approximation theorem for conformal minimal immersions of open 
Riemann surfaces into $\r^n$ for any $n\ge 3$; see Theorem \ref{th:Mergelyan}. The special case $n=3$ has been
already proved by  Alarc\'on and L\'opez \cite[Theorem 4.9]{AL1} by using the {\em L\'opez-Ros transformation} 
for conformal minimal immersions $M\to\r^3$ (see \cite{LR}), a tool that is not available for $n\geq 4$. 
Here we use the more general approach which has been developed in \cite{AF2} 
for approximating holomorphic null curves and more general directed holomorphic immersions 
of open Riemann surfaces to $\c^n$.

We begin by introducing a suitable type of sets  for the Mergelyan approximation. 
The same type of sets have been used in \cite{AF2} (see Definition 7.1 there) and
in several other papers.

\begin{definition}
\label{def:admissible}
A compact subset $S$ of an open Riemann surface $M$ is said to be {\em admissible} if $S=K\cup \Gamma$,
where $K=\bigcup \overline D_j$ is a union of finitely many pairwise disjoint, compact, smoothly bounded domains 
$\overline D_j$ in $M$ and $\Gamma=\bigcup \Gamma_i$ is a union of finitely many pairwise disjoint smooth arcs or closed 
curves that intersect $K$ only in their endpoints (or not at all), and such that their intersections with the boundary $bK$ are transverse. 
\end{definition}

Note that an admissible set $S\subset M$ is {\em Runge} in $M$ (i.e.,\ $\Oscr(M)$-convex)
if and only if the inclusion map $S\hra M$ induces an injective homomorphism 
$H_1(S;\z)\hra H_1(M;\z)$ of the first homology groups. If this holds, then we have the 
Mergelyan approximation theorem: Every continuous function $f\colon S=K\cup\Gamma\to \c$ 
that is holomorphic in the interior $\mathring K$ of the compact set $K$ can be approximated, 
uniformly on $S$, by functions holomorphic on $M$. 
More generally, if $f$ is of class ${\mathscr C}^r$ on $S$ for some
$r\ge 0$, then the approximation can be made in the ${\mathscr C}^r(S)$ topology. 

Recall that $\Agot$ denotes the null quadric (\ref{eq:Agot}) and $\Agot^*=\Agot\setminus\{0\}$.

Given an admissible set $S=K\cup \Gamma\subset M$, we denote by $\Ogot(S,\Agot^*)$ the set of all 
smooth maps $S\to\Agot^*$  which are holomorphic on an unspecified open neighborhood of $K$ 
(depending on the map). In accordance with Definition \ref{def:nondegenerate},
we say that a map $f\in\Ogot(S,\Agot^*)$ is {\em nonflat} if it maps no component of $K$ and no component of $\Gamma$ to a ray in $\Agot^*$. Likewise, we say that $f\in\Ogot(S,\Agot^*)$ is {\em nondegenerate} if it maps no component of $K$ and no component of $\Gamma$ to a complex hyperplane of $\c^n$. We denote by $\Ogot_*(S,\Agot^*)$ the subset of $\Ogot(S,\Agot^*)$ consisting of all nondegenerate maps.

Fix a nowhere vanishing holomorphic 1-form $\theta$ on $M$. 
The following notion of a generalized conformal minimal immersion  emulates the spirit of the concept of 
{\em marked immersion} \cite{AL1}. (The same notion has been used in \cite[Definition 6.2]{AF4}.)

%
%
%
%
\begin{definition}\label{def:generalized}
Let $M$ be an open Riemann surface and let $S=K\cup \Gamma \subset M$ be an admissible subset 
(see Definition \ref{def:admissible}). A {\em generalized conformal minimal immersion} on $S$ 
is a pair $(u,f\theta)$, where $f\in \Ogot(S,\Agot^*)$
and $u\colon S\to\r^n$ is a smooth map which is a conformal minimal immersion on an open 
neighborhood of $K$, such that the following properties hold:
\begin{itemize}
\item $f\theta = 2\partial u$ on an open neighborhood of $K$ in $M$;

\vspace{1mm}

\item for any smooth path $\alpha$  in $M$ parametrizing a connected component of $\Gamma$ we have
$\Re(\alpha^*(f\theta)) = \alpha^*(du)= d(u\circ\alpha)$.
\end{itemize}
A generalized conformal minimal immersion $(u,f\theta)$ is said to be {\em nonflat} 
(resp.\ {\em nondegenerate}) if the map $f\in\Ogot(S,\Agot^*)$ is nonflat (resp.\ nondegenerate)
on every connected component of $K$ and on every connected component
of $\Gamma$ (see Definition \ref{def:nondegenerate}). 
\end{definition}

Property (b) shows that a generalized conformal minimal immersion 
on a curve $C\subset M$ is nothing else than a $1$-jet of a conformal immersion along $C$. 

We denote by $\GCMI(S,\r^n)$ the set of all generalized conformal minimal immersions $S\to\r^n$ and by 
$\GCMI_*(S,\r^n)\subset \GCMI(S,\r^n)$ the subset consisting of all nondegenerate ones.
We say that $(u,f\theta)\in\GCMI(S)$ {\em can be approximated in the $\Cscr^1(S)$ topology} by conformal minimal immersions 
in $\CMI(M)$ if there exists a sequence $v_i\in \CMI(M)$ $(i\in\n)$ such that $v_i|_S$ converges to $u|_S$ 
in the $\Cscr^0(S)$ topology and $2\partial v_i|_S$ converges to $f\theta|_S$ in the $\Cscr^0(S)$ topology. 

%
%
%
%
%
\begin{theorem}[Mergelyan's theorem for conformal minimal immersions]
\label{th:Mergelyan}
Assume that $M$ is an open Riemann surface and that $S=K\cup \Gamma$ is a compact Runge admissible set in $M$. 
Then, every generalized conformal minimal immersion $(u,f\theta)\in\GCMI(S, \r^n)$  for $n\ge 3$ can be approximated in 
the $\Cscr^1(S)$ topology by nondegenerate conformal minimal immersions $\tilde u\in\CMI_*(M,\r^n)$. 

Furthermore, given a group homomorphism $\pgot \colon H_1(M;\z)\to\r^n$ satisfying $\pgot(C)= \Flux_{u} (C)$ for 
every closed curve $C\subset S$, we can choose $\tilde u$ as above such that $\Flux_{\tilde u}=\pgot$.
\end{theorem}

\begin{proof}
Let $\rho:M\to\r$ be a smooth strongly subharmonic Morse exhaustion function.
We  exhaust $M$ by an increasing sequence 
$M_1\subset M_2\subset\cdots\subset \bigcup_{i=1}^\infty M_i=M$
of compact smoothly bounded domains of the form $M_i=\{p\in M\colon \rho(p)\le c_i\}$,
where $c_1<c_2<\cdots$ is an increasing sequence of regular values of $\rho$
with $\lim_{i\to\infty} c_i =+\infty$. Each domain $M_i$ is therefore a compact bordered Riemann surface,
possibly disconnected. We may assume that $\rho$ has at most one critical
point $p_i$ in each difference $M_{i+1}\setminus M_i$. It then follows that $M_{i}$ is
 Runge in $M$ for every $i\in \n$. Finally, since $S$ is Runge, we may assume without 
 loss of generality that $S\subset \mathring M_1$ and $S$ is a strong deformation retract of $M_1$, 
 hence the inclusion map $S\hookrightarrow M_1$ induces an isomorphism
$H_1(S;\z)\cong H_1(M_1;\z)$ of the homology groups.

We proceed by induction. The basis is given by the following lemma.

\begin{lemma}\label{lem:step1}
Every $(u,f\theta)\in\GCMI(S, \r^n)$  for $n\ge 3$ can be approximated in the $\Cscr^1(S)$ topology 
by nondegenerate conformal minimal immersions in $\CMI_*(M_1,\r^n)$.
\end{lemma}

\begin{proof}
Since $S=K\cup \Gamma$ is a strong deformation retract of $M_1$, 
we may assume that $S$ is connected; the same argument can be applied on any connected component.

By part (a) of Theorem \ref{th:local}, and  deforming $(u,f\theta)$ slightly on 
$\Gamma$, we may assume that $(u,f\theta)\in\GCMI_*(S,\r^n)$, 
i.e., it is nondegenerate in the sense of Definition \ref{def:generalized}.

{\em Claim:} It is possible to approximate $f\in \Ogot_*(S,\Agot^*)$ as closely as desired 
uniformly on $S$ by a holomorphic map $f_1:M_1\to\Agot^*$ such that 
\begin{equation}\label{eq:sameperiod}
	\int_C (f_1-f)\theta=0 \quad  \text{for every closed curve} \ C\subset S.  
\end{equation}
Assume for a moment that this holds.
Since $H_1(S;\z)\cong H_1(M_1;\z)$ and $f\theta$ has no real periods on $S$, the 
same is true for $f_1$ on $M_1$ in view of (\ref{eq:sameperiod}). Hence, $f_1$
provides a conformal minimal immersion $u_1\in \CMI(M_1,\r^n)$ by the expression 
\[
	u_1(p)=u(p_0)+\int_{p_0}^p \Re (f_1\theta), \quad p\in M_1, 
\]
where $p_0\in K$ is any base point. Furthermore, since $S$ is connected, 
$u_1$ can be assumed to be as close as desired to $u$ in the $\Cscr^1(S)$ topology 
provided that the approximation of $f$ by $f_1$ 
is close enough. In particular, since $u$ is nondegenerate, 
$u_1$ can be taken in $\CMI_*(M_1,\r^n)$. 

This proves Lemma \ref{lem:step1} provided  that the above claim holds.

The construction of a holomorphic map $f_1\colon M_1\to\Agot^*$ satisfying 
(\ref{eq:sameperiod}) is similar to the proof of the 
Mergelyan approximation theorem for null holomorphic curves (and other classes 
of directed holomorphic immersions) in \cite[Theorem 7.2]{AF2}. 
The main difference is that the period vanishing condition in the latter result 
is now replaced by the condition of matching the periods of a given map.
Here is the outline.

By the assumption, the map $f$ is holomorphic on an open neighborhood $U\subset M$ 
of $K$ and is smooth on $\Gamma$. By part (a) of Theorem \ref{th:local}, we may assume that 
$f$ is nondegenerate. Up to a shrinking of $U$ around $K$, we can apply \cite[Lemma 5.1]{AF2} 
to find a period dominating spray of 
smooth maps $f_w\colon U\cup \Gamma \to \Agot^*$ which are holomorphic on $U$
and depend holomorphically on a parameter $w$ in a ball $B\subset \c^N$, with $f_0=f$.
(One deforms $f$ by flows of holomorphic vector fields on $\Agot$
which generate the tangent space at every point of $\Agot^*$; see (\ref{eq:Phi})
and (\ref{eq:Phi2}) above. The time variables of these flows are holomorphic functions 
on a neighborhood of $S$ in $M$, chosen so as to ensure the period domination property.)

By Mergelyan approximation,  we can approximate the map $f=f_0$ uniformly on $S$ by  a map 
$\tilde f_0$ which is holomorphic on a small open neighborhood 
$V\subset M$ of the set $S$. (Explicitly, we can use \cite[Theorem 3.7.2, p.\ 81]{F2011},
noticing that our set $S$ is a special case of the sets $S=K_0\cup E$ in the cited theorem.)
By applying the same flows to $\tilde f_0$, we get a new holomorphic spray of maps 
$\tilde f_w\colon V\to\Agot^*$ which approximates the initial spray $f_w$ uniformly
on $S$, and uniformly with respect to the parameter $w$.  
(This part of the construction can be done with an arbitrary complex 
manifold $X$ in place of $\Agot^*$.) 

Since $\Agot^*$ is an Oka manifold (see Remark \ref{rem:Oka})
and $S$ is Runge in $M$ and a deformation retract of $M_1$, 
we can apply \cite[Theorem 5.4.4, p.\ 193]{F2011} (the main result of Oka theory)
to approximate the spray $\tilde f_w$ uniformly on $S$ (and uniformly with respect to the
parameter $w$) by a holomorphic spray of maps $g_w \colon M_1\to X$. 
(The parameter set $B\subset \c^N$ is allowed to shrink a little.  
The topological condition on the inclusion $S\hra M_1$ is used to get the existence of a 
continuous extension of the spray $\tilde f_w$ from $S$ to $M_1$,  a necessary condition 
to apply the Oka principle.) 

If both approximations made above are close enough, then there exists a point $w_0\in B$ close to
the origin such that the map $g_{w_0}\colon M_1\to \Agot^*$ satisfies the period condition (\ref{eq:sameperiod}). 
(The last argument is as in the proof of Theorem \ref{th:local} {\rm (c)}.)
Taking this map as our $f_1$ completes the proof of Lemma \ref{lem:step1}.
\end{proof}

The following result provides the inductive step in the recursive process.
\begin{lemma}\label{lem:step2}
Assume that $\pgot\colon H_1(M;\z)\to\r^n$ is a group homomorphism.
Let $i\in\n$ and let $u_i\in\CMI_*(M_i,\r^n)$ be a nondegenerate conformal minimal immersion such that 
$\Flux_{ u_i}(C)=\pgot(C)$ for all closed curve $C\subset M_i$. Then, $u_i$ can be approximated in the 
$\Cscr^1(M_i)$ topology by nondegenerate conformal minimal immersions $ u_{i+1}\in\CMI_*(M_{i+1},\r^n)$ 
satisfying $\Flux_{ u_{i+1}}(C)=\pgot(C)$ for all closed curve $C\subset M_{i+1}$.
\end{lemma}

\begin{proof}
We consider two essentially different cases.

\smallskip
\noindent{\em The noncritical case}: $\rho$ has no critical value in $[c_i,c_{i+1}]$. In this case, 
there is no change of topology when passing from $M_i$ to $M_{i+1}$, and $M_i$ is a strong deformation 
retract of $M_{i+1}$. The immersion $u_{i+1}$ can then be constructed as in the proof of Lemma \ref{lem:step1}.

\smallskip
\noindent{\em The critical case}: $\rho$ has a critical point $p_{i+1}\in M_{i+1}\setminus M_i$. 
By the assumptions on $\rho$, $p_{i+1}$ is the only critical point of $\rho$ on $M_{i+1}\setminus M_i$ and is a Morse point. 
Since $\rho$ is strongly subharmonic, the Morse index of $p_{i+1}$ is either $0$ or $1$. 

If the Morse index of $p_{i+1}$ is $0$, then a new (simply connected) component of the sublevel set $\{\rho\leq r\}$ 
appears at $p_{i+1}$ when $r$ passes the value $\rho(p_{i+1})$. In this case, we reduce the proof to the noncritical case 
by defining $u_{i+1}$ on this new component as any nondegenerate conformal minimal immersion.

Assume now that the Morse index of $p_{i+1}$ is $1$. In this case, the change of topology of the sublevel set 
$\{\rho\leq r\}$ at $p_{i+1}$ is described by attaching to $M_i$ a smooth
arc $\gamma\subset \mathring M_{i+1}\setminus M_i$ such that $M_i\cup \gamma$ is a 
Runge strong deformation retract of $M_{i+1}$. We assume without loss of generality that $M_i\cup \gamma$ is admissible 
(see Definition \ref{def:admissible}). Since $u_i$ is nondegenerate and $\Flux_{u_i}(C)=\pgot(C)$ for all closed curve 
$C\subset M_i$, we may extend $u_i$ to $M_i\cup\gamma$ as a nondegenerate generalized conformal minimal immersion 
$(\hat u_i,\hat f_i\theta)\in\GCMI_*(M_i\cup\gamma,\r^n)$ such that $\hat u_i=u_i$ on $M_i$ and 
$\int_C \Im(f_i\theta) =\pgot(C)$ for all closed curve $C\subset M_i\cup\gamma$.
This can be done as in \cite[Lemma 3.4]{AF4} where the details are given for the case $n=3$,
but the same proof works in general. 
By Lemma \ref{lem:step1}, we can approximate $(\hat u_i,\hat f_i\theta)$ in $\Cscr^1(M_i\cup\gamma)$ by a conformal
minimal immersion on an open neighborhood of $M_i\cup\gamma$ without changing the flux. This reduces the proof
to the noncritical case considered above.
\end{proof}

Combining Lemmas \ref{lem:step1} and \ref{lem:step2}, we may construct a sequence of nondegenerate 
conformal minimal immersions $\{ u_i\in\CMI_*(M_i)\}_{i\in\n}$ such that:
\begin{itemize}
\item $ u_i$ is as close  to $(u,f\theta)$ as desired in the $\Cscr^1(S)$ topology for all $i\in\n$.
\item $ u_i$ is as close to $ u_{i-1}$ as desired in the $\Cscr^1(M_{i-1})$ topology  for all $i\geq 2$.
\item $\Flux_{ u_i}(C)=\pgot(C)$ for all closed curve $C\subset M_i$ and all $i\in\n$.
\end{itemize}
If these approximations are close enough, then the limit $\tilde u:=\lim_{i\to\infty} u_i:M\to\r^n$ is a nondegenerate conformal 
minimal immersion as close to $(u,f\theta)$ in the $\Cscr^1(S)$ topology as desired and satisfying $\Flux_{\tilde u}=\pgot$. 
This concludes the proof of Theorem \ref{th:Mergelyan}.
\end{proof}

The following Mergelyan type result for conformal minimal immersions into $\r^n$ with $n-2$ fixed components 
was essentially proved in \cite{AFL2}. It will play an important role in the proof of Theorem \ref{th:proper} in Section  \ref{sec:proper}.

\begin{lemma} 
\label{lem:Mergelyan2}
Assume that $M$ is a compact bordered Riemann surface and $K$ is a union of finitely 
many pairwise disjoint, smoothly bounded, compact Runge domains in $\mathring M$. 
Assume that $K$ contains a basis of $H_1(M;\z)$ and naturally identify $H_1(K;\z)=H_1(M;\z)$. 
Let $u=(u^1,\ldots,u^n)\in\CMI_*(K,\r^n)$ be a nondegenerate conformal minimal immersion  
and assume that $u^j$ extends harmonically to $M$ for all $j\geq 3$.
Then $u$ can be approximated in the $\Cscr^1(K)$ topology by nondegenerate conformal minimal immersions 
$\tilde u=(\tilde u^1,\ldots,\tilde u^n)\in\CMI_*(M,\r^n)$ such that $\Flux_{\tilde u}=\Flux_u$ and $\tilde u^j=u^j$ for all $j\geq 3$.
\end{lemma}

\begin{proof} 
Let $\theta$ be a nowhere vanishing holomorphic $1$-form on $M$. As usual, we write $f^j=2 \partial u^j/\theta\in  \Oscr(K)$ 
for all $j$ and notice that $f^j\in  \Oscr(M)$ for all $j\geq 3$. Denote by $\Theta$ the quadratic holomorphic form 
$-(\sum_{j=3}^n (f^j)^2) \theta^2$ on $M$. 

Let $S=K\cup\Gamma\subset M$ be a Runge connected admissible set (see Definition \ref{def:admissible}) 
such that $\Theta$ does not vanish anywhere on $\Gamma$. Choose a nondegenerate generalized conformal 
minimal immersion $(v,g\theta)\in\GCMI_*(S,\r^n)$ such that
\begin{itemize}
\item $v=u$ and $g=f$ on $K$,
\item $v^j=u^j$ and $g^j=f^j$ on $S$ for all $j\geq 3$, where $v=(v^1,\ldots,v^n)$ and $g=(g^1,\ldots,g^n)$.
\end{itemize}
(We refer the reader to \cite[Proof of Lemma 3.3]{AFL2} for details on how to find such $(v,g\theta)$.) 
By the latter condition, $(g^1)^2+(g^2)^2=-\sum_{j=3}^n (g^j)^2=\Theta/\theta^2$ on $S$. 
Further, since $u$ is nondegenerate, the functions $g^1$ and $g^2$ are lineraly independent in $\Oscr(K)$. 
By \cite[Lemma 3.3]{AFL2}, 
$(g^1,g^2)$  can be uniformly approximated on $S$ by a pair $(h^1,h^2) \subset \Oscr(M)^2$ satisfying
\begin{itemize}
\item $(h^1)^2 + (h^2)^2 = \Theta/\theta^2$,
\item the $1$-form $\big((h^1,h^2)-(g^1,g^2)\big)\theta$ is exact on $S$, and
\item the zeros of $(h^1,h^2)$ on $M$ are those of $(f^1,f^2)$ on $K$ 
(in particular, $(h^1,h^2)$ does not vanish anywhere on $M\setminus K$).
\end{itemize}

Fix a point $p_0\in K$ and set $\tilde u^j(p):= u^j(p_0)+\Re \int_{p_0}^p h^j\theta$, $p\in M$, $j=1,2$, and  $\tilde u^j:=u^j$ 
for all $j=3,\ldots,n$. If the approximation of $(g^1,g^2)$ by $(h^1,h^2)$ is close enough on $S$, then it is clear that 
$\tilde u=(\tilde u^1,\ldots,\tilde u^n)\in\CMI_*(M,\r^n)$ satisfies the conclusion of the lemma.
\end{proof}

%

\section{Approximation by conformal minimal embeddings}
\label{sec:th1.1}
In this section, we prove the following more precise version of Theorem \ref{th:desing}.

\begin{theorem}
\label{th:desing2}
Let $M$ be an open Riemann surface and let $n\ge 5$ be an integer.
Given a conformal minimal immersion $u\colon M\to \r^n$, a compact Runge set 
$K\subset M$ and a number $\epsilon>0$, there exists a conformal minimal embedding 
$\tilde u\colon M\to \r^n$ such that $\sup_{x\in K}|\tilde u(x) - u(x)|<\epsilon$ 
and $\Flux(\tilde u)=\Flux(u)$.
\end{theorem}

\begin{proof}
Exhaust $M$ by an increasing sequence $M_1\subset M_2\subset\cdots\subset \bigcup_{i=1}^\infty M_i=M$ of compact,
smoothly bounded,  Runge  domains such that $K\subset M_1$ and each domain $M_i$ is a compact bordered Riemann surface.

We proceed by induction.

Theorem \ref{th:desingBRS} furnishes a conformal minimal embedding $u_1\in\CMI(M_1,\r^n)$ which is as close as desired to $u$ in 
the $\Cscr^1(M_1)$ topology and satisfies $\Flux_{u_1}(C)=\Flux_u(C)$ for all closed curve $C\subset M_1$. 
We may further assume by Theorem \ref{th:local} {\rm (a)} that $u_1$ is nondegenerate, i.e., $u_1\in \CMI_*(M_1,\r^n)$. 

Let $i\in\n$ and assume the existence of a nondegenerate conformal minimal embedding $u_i\in\CMI_*(M_i,\r^n)$ 
with $\Flux_{u_i}(C)=\Flux_u(C)$ for all closed curve $C\subset M_i$. Theorem \ref{th:Mergelyan} ensures that $u_i$ can 
be approximated in the $\Cscr^1(M_i)$ topology by nondegenerate conformal minimal immersions 
$u_{i+1}\in\CMI_*(M_{i+1},\r^n)$ 
with $\Flux_{u_{i+1}}(C)=\Flux_u(C)$ for all closed curve $C\subset M_{i+1}$. Moreover, in view of Theorem \ref{th:desingBRS}, 
this approximation can be done by embeddings.

This process gives a sequence of nondegenerate conformal minimal embeddings $\{u_i\in\CMI_*(M_i,\r^n)\}_{i\in\n}$ 
such that
\begin{itemize}
\item $u_i$ is as close to $u$ as desired in the $\Cscr^1(K)$ topology for all $i\in\n$.
\item $u_i$ is as close to $u_{i-1}$ as desired in the $\Cscr^1(M_{i-1})$ topology  for all $i\geq 2$.
\item $\Flux_{u_i}(C)=\Flux_u(C)$ for all closed curve $C\subset M_i$ and all $i\in\n$.
\end{itemize}
If these approximations are close enough, then the limit $\tilde u=\lim_{i\to\infty} u_i:M\to\r^n$ 
is a nondegenerate conformal minimal embedding satisfying the conclusion of the theorem.
(See for instance \cite[Proof of Theorem 4.5]{AL-C2} for a similar argument.)
\end{proof}


\section{Construction of proper conformal minimal embeddings}
\label{sec:proper}

In this section we prove Theorem \ref{th:proper} in the following more precise form.

\begin{theorem}
\label{th:proper2}
Let $M$ be an open Riemann surface and let $K\subset M$ be a compact smoothly bounded Runge 
domain in $M$. Let $u=(u^1,\ldots,u^n):K\to\r^n$ be a conformal minimal immersion on a neighborhood of $K$ and 
let $\pgot: H_1(M;\z)\to\r^n$ be a group homomorphism satisfying $\pgot(C)= \Flux_{u} (C)$ for every closed curve 
$C\subset K$.  Then, for any $\epsilon>0$, there exists a nondegenerate conformal minimal immersion 
$\tilde u=(\tilde u^1,\ldots,\tilde u^n):M\to \r^n$ such that 
$\sup_{x\in K}|\tilde u(x) - u(x)|<\epsilon$, $(\tilde u^1,\tilde u^2):M\to\r^2$ is proper,
and $\Flux_{\tilde u}=\pgot$. Furthermore, if $n\geq 5$, the approximating immersions 
$\tilde u:M\to\r^n$ can be taken to be embeddings.
\end{theorem}

Theorem \ref{th:proper2} was already proved for $n=3$ by Alarc\'on and L\'opez in \cite{AL1}.
Their proof uses the Weierstrass representation of conformal minimal immersions $M\to\r^3$ 
and hence it does not generalize to the case $n>3$.

\begin{proof}
Let $\rho:M\to\r$ be a smooth strongly subharmonic Morse exhaustion function and 
exhaust $M$ by an increasing sequence 
$M_1\subset M_2\subset\cdots\subset \bigcup_{i=1}^\infty M_i=M$
of compact smoothly bounded Runge domains of the form $M_i=\{p\in M\colon \rho(p)\le c_i\}$,
where $c_1<c_2<\cdots$ is an increasing sequence of regular values of $\rho$
with $\lim_{i\to\infty} c_i =+\infty$. Each domain $M_i$ is a compact bordered Riemann surface,
possibly disconnected. Assume that $\rho$ has at most one critical
point $p_i$ in each difference $M_{i+1}\setminus M_i$. Since $K$ is Runge in $M$, 
we may also assume that $M_1=K$.

Since $K$ is compact, we may assume up to a translation that $\max\{u^1,u^2\}>1$ on $bK$. 
Further, by Theorem \ref{th:local} {\rm (a)} we may assume that $u$ is nondegenerate 
(see Definition \ref{def:nondegenerate}).

We proceed by induction. The initial immersion is $u_1=u\in\CMI_*(M_1,\r^n)$. 
The inductive step is furnished be the following lemma.

\begin{lemma}
\label{lem:proper}
Let $i\in\n$ and let $u_i=(u_i^1,\ldots,u_i^n)\in\CMI_*(M_i,\r^n)$ be a nondegenerate conformal 
minimal immersion such that 
\begin{enumerate}[\rm (I)]
\item $\Flux_{u_i}(C)=\pgot(C)$ for all closed curve $C\subset M_i$ and
\item $\max\{u_i^1,u_i^2\}>i$ on $bM_i$.
\end{enumerate}
Then $u_i$ can be approximated in the $\Cscr^1(M_i)$ topology by nondegenerate 
conformal minimal immersions $u_{i+1}=(u_{i+1}^1,\ldots,u_{i+1}^n)\in\CMI_*(M_{i+1},\r^n)$ such that 
\begin{enumerate}[\rm (i)]
\item $\Flux_{u_{i+1}}(C)=\pgot(C)$ for all closed curve $C\subset M_{i+1}$,
\item $\max\{u_{i+1}^1,u_{i+1}^2\}>i$ on $M_{i+1}\setminus \mathring M_i$, and
\item $\max\{u_{i+1}^1,u_{i+1}^2\}>i+1$ on $bM_{i+1}$.
\end{enumerate}
\end{lemma}
\begin{proof}
By Theorem \ref{th:local} {\rm (c)} we may assume that $u_i$ extends as a conformal minimal immersion 
to an unspecified open neighborhood of $M_i$.

We consider two essentially different cases.

\smallskip
\noindent{\em The noncritical case}: Assume that $\rho$ has no critical value in $[c_i,c_{i+1}]$.

In this case there is no change of topology when passing from $M_i$ to $M_{i+1}$ and $M_i$ is a strong deformation 
retract of $M_{i+1}$. Denote by $m\in\n$ the number of boundary components of $bM_i$. It follows that 
$M_{i+1}\setminus \mathring M_i=\bigcup_{j=1}^m A_j$ where the sets $A_j$, $j=1,\ldots,m$, are pairwise 
disjoint smoothly bounded compact annuli. For each $j\in\{1,\ldots,m\}$ write 
$bA_j=\alpha_j\cup\beta_j$ with $\alpha_j\subset bM_i$ and $\beta_j\subset bM_{i+1}$. 
In view of condition {\rm (II)} in the statement of 
the lemma, there exists an integer $l\geq 3$ such that each $\alpha_j$ splits into $l$ compact subarcs 
$\alpha_{j,k}$, $k\in\z_l=\z/l\z$, satisfying the following conditions:
\begin{enumerate}[\rm ({a}1)]
\item $\alpha_{j,k}$ and $\alpha_{j,k+1}$ intersect at a common endpoint $p_{j,k}$ and 
$\alpha_{j,k}\cap\alpha_{j,a}=\emptyset$ for all 
$a\in\z_l\setminus\{k-1,k,k+1\}$, for all $(j,k)\in I:=\{1,\ldots,m\}\times\z_l$.
\item $\bigcup_{k\in\z_l}\alpha_{j,k}=\alpha_j$ for all $j\in\{1,\ldots,m\}$.
\item There exist subsets $I_1,I_2$ of $I$ such that $I=I_1\cup I_2$, $I_1\cap I_2=\emptyset$, 
and $u_i^\sigma>i$ on $\alpha_{j,k}$ for all $(j,k)\in I_\sigma$, $\sigma=1,2$.
\end{enumerate}

For every $j=1,\ldots,m$ let $\{\gamma_{j,k}\subset A_j\colon (j,k)\in I\}$ be a family of pairwise disjoint 
smooth Jordan arcs such that $\gamma_{j,k}$ connects $p_{j,k}\in \alpha_j$ with a point 
$q_{j,k}\in \beta_j$ and is otherwise disjoint with $bA_j$, . We may assume in addition that the set
\[
	S=M_i\cup \bigcup_{(j,k)\in I} \gamma_{j,k}
\]
is admissible in the sense of Definition \ref{def:admissible}. Recall that $M_i$ is Runge in $M$, 
and hence $S\subset M$ is Runge as well. 
Let $\theta$ be a nowhere vanishing holomorphic $1$-form on $M$. Extend $(u_i,2\partial u_i)$ 
to $S$ as a nondegenerate generalized conformal minimal immersion $(u_i,f\theta)\in\GCMI_*(S,\r^n)$ satisfying that:
\begin{itemize}
\item $u_i^\sigma>i$ on $\gamma_{j,k}\cup\alpha_{j,k}\cup\gamma_{j,k-1}$ for all $(j,k)\in I_\sigma$, $\sigma=1,2$.
\item $u_i^\sigma>i+1$ on $\{q_{j,k},q_{j,k-1}\}$ for all $(j,k)\in I_\sigma$, $\sigma=1,2$.
\end{itemize}
The existence of such extension is trivially ensured by property {\rm (a3)}. 
Theorem \ref{th:Mergelyan} then provides 
$v=(v^1,\ldots,v^n)\in\CMI_*(M_{i+1},\r^n)$ enjoying the following properties:
\begin{enumerate}[\rm ({b}1)]
\item $v$ is as close as desired to $u_i$ in the $\Cscr^1(M_i)$ topology.
\item $v^\sigma>i$ on $\gamma_{j,k}\cup\alpha_{j,k}\cup \gamma_{j,k-1}$ for all $(j,k)\in I_\sigma$, $\sigma=1,2$.
\item $v^\sigma>i+1$ on $\{q_{j,k},q_{j,k-1}\}$ for all $(j,k)\in I_\sigma$, $\sigma=1,2$.
\item $\Flux_v(C)=\Flux_{u_i}(C)$ for any closed curve $C\subset M_i$.
\end{enumerate}

Denote by $\beta_{j,k}$ the subarc of $\beta_j$ connecting $q_{j,k-1}$ and $q_{j,k}$ and containing $q_{j,a}$ for no 
$a\in\z_l\setminus\{k-1,k\}$, for all $(j,k)\in I$. Denote by $\Omega_{j,k}\subset A_j$ the closed disc bounded by $\gamma_{j,k-1}$, 
$\alpha_{j,k}$, $\gamma_{j,k}$, and $\beta_{j,k}$, $(j,k)\in I$. By {\rm (b2)}, {\rm (b3)} and the continuity of $v$ 
there exist compact, smoothly bounded discs 
$D_{j,k}\subset \Omega_{j,k}\setminus (\gamma_{j,k-1}\cup\alpha_{j,k}\cup\gamma_{j,k})$, $(j,k)\in I$, 
such that $D_{j,k}\cap\beta_{j,k}\neq\emptyset$ is a subarc of $\beta_{j,k}\setminus\{q_{j,k-1},q_{j,k}\}$ 
and the following conditions hold:
\begin{enumerate}[\rm ({b}1')]
\item[\rm (b2')] $v^\sigma>i$ on $\overline{\Omega_{j,k}\setminus D_{j,k}}$ for all $(j,k)\in I_\sigma$, $\sigma=1,2$.
\item[\rm (b3')] $v^\sigma>i+1$ on $\overline{\beta_{j,k}\setminus D_{j,k}}$ for all $(j,k)\in I_\sigma$, $\sigma=1,2$.
\end{enumerate}

Assume that $I_1\neq \emptyset$, otherwise $I_2\neq \emptyset$ and we would reason in a symmetric way.

Consider the compact smoothly bounded Runge domain
\[
	S_1=M_1\cup \Big( \bigcup_{(j,k)\in I_2} \Omega_{j,k}\Big) \cup \Big( \bigcup_{(j,k)\in I_1} D_{j,k}\Big) .
\]
Observe that $S_1$ is not connected; its components are 
$M_1\cup \bigcup_{(j,k)\in I_2} \Omega_{j,k}$ and $D_{j,k}$, $(j,k)\in I_1$. 
Since $D_{j,k}$ is compact, there exists a constant $\tau_1>0$ such that
\begin{equation}\label{eq:tau1}
\tau_1+v^2>i+1\quad \text{on $\bigcup_{(j,k)\in I_1} D_{j,k}$},
\end{equation}
recall that $v=(v^1,v^2,\ldots,v^n)$.

Denote by $\hat v_1=(\hat v_1^1,\hat v_1^2,\ldots,\hat v_1^n)\in\CMI_*(S_1,\r^n)$ 
the conformal minimal immersion given by 
\begin{equation}\label{eq:hatv11}
\hat v_1= v\quad \text{on $M_1\cup \bigcup_{(j,k)\in I_2} \Omega_{j,k}$,} 
\end{equation}
\begin{equation}\label{eq:hatv12}
\hat v_1=(v^1,\tau_1+v^2,\ldots,v^n)\quad \text{on $\bigcup_{(j,k)\in I_1} D_{j,k}$.}
\end{equation}
Observe that every component of $\hat v_1$ equals the restriction to $S_1$ of the corresponding component of $v$, 
except for $\hat v_1^2$. 
By Lemma \ref{lem:Mergelyan2} we may approximate $\hat v_1$ in the $\Cscr^1(S_1)$ topology by a nondegenerate 
conformal minimal immersion $v_1=(v_1^1,v_1^2,\ldots,v_1^n)\in\CMI_*(M_{i+1},\r^n)$ satisfying the following properties:
\begin{enumerate}[\rm ({c}1)]
\item $v_1$ is as close as desired to $u_i$ in the $\Cscr^1(M_i)$ topology.
\item $v_1^1=v^1$ on $M_{i+1}$.
\item $v_1^a>i$ on $\bigcup_{(j,k)\in I_a}\overline{\Omega_{j,k}\setminus D_{j,k}}$, $a=1,2$. 
Take into account {\rm (b2')}, \eqref{eq:hatv11}, and {\rm (c2)}.
\item $v_1^a>i+1$ on $\bigcup_{(j,k)\in I_a}\overline{\beta_{j,k}\setminus D_{j,k}}$, $a=1,2$.  
See {\rm (b3')}, \eqref{eq:hatv11},  and {\rm (c2)}.
\item $v_1^2>i+1$ on $\bigcup_{(j,k)\in I_1} D_{j,k}$. Take into account \eqref{eq:hatv12} and \eqref{eq:tau1}.
\item $\Flux_{v_1}(C)=\Flux_{u_i}(C)$ for any closed curve $C\subset M_i$. See {\rm (b4)}.
\end{enumerate}

Assume that $I_2\neq \emptyset$; otherwise the immersion $u_{i+1}=v_1$ satisfies the conclusion of the lemma 
and we are done.  Indeed, if $I_2=\emptyset$, then $I_1=I$ and we have
\[
	M_{i+1}\setminus M_i\subset \big(\bigcup_{(j,k)\in I_1}\overline{\Omega_{j,k} 
	\setminus D_{j,k}}\big) \cup \big( \bigcup_{(j,k)\in I_1} D_{j,k}\big)
\]
and
\[
	bM_{i+1}\subset \big(\bigcup_{(j,k)\in I_1}\overline{\beta_{j,k}\setminus D_{j,k}}\, \big) 
	\cup \big( \bigcup_{(j,k)\in I_1} D_{j,k}\big).
\]
Therefore, properties {\rm (c3)} and {\rm (c5)} above imply Lemma \ref{lem:proper} {\rm (ii)}, 
whereas {\rm (c4)} and {\rm (c5)} 
ensure {\rm (iii)}. Finally {\rm (c6)}, Lemma \ref{lem:proper} {\rm (I)}, and the fact that $M_i$ 
is a strong deformation retract 
of $M_{i+1}$ give {\rm (iii)}. This and {\rm (c1)} would conclude the proof.

Consider the compact Runge, smoothly bounded domain
\[
S_2=M_2\cup \Big( \bigcup_{(j,k)\in I_1} \Omega_{j,k}\Big) \cup \Big( \bigcup_{(j,k)\in I_2} D_{j,k}\Big) .
\]
Since $I_2\neq\emptyset$, $S_2$ is not connected. Pick a constant $\tau_2>0$ such that
\begin{equation}\label{eq:tau2}
\tau_2+v_1^1>i+1\quad \text{on $\bigcup_{(j,k)\in I_2} D_{j,k}$}.
\end{equation}

Define $\hat v_2=(\hat v_2^1,\hat v_2^2,\ldots,\hat v_2^n)\in\CMI_*(S_1,\r^n)$ by 
\begin{equation}\label{eq:hatv112}
\hat v_2= v_1\quad \text{on $M_1\cup \bigcup_{(j,k)\in I_1} \Omega_{j,k}$,} 
\end{equation}
\begin{equation}\label{eq:hatv122}
\hat v_2=(\tau_2+v_1^1,v_1^2,\ldots,v_1^n)\quad \text{on $\bigcup_{(j,k)\in I_2} D_{j,k}$.}
\end{equation}
Now every component of $\hat v_2$ equals the restriction to $S_2$ of the corresponding component of $v_1$, 
except for $\hat v_2^1$. 
By Lemma \ref{lem:Mergelyan2} we may approximate $\hat v_2$ in the $\Cscr^1(S_2)$ 
topology by an immersion $v_2=(v_2^1,v_2^2,\ldots,v_2^n)\in\CMI_*(M_{i+1},\r^n)$ such that:
\begin{enumerate}[\rm ({d}1)]
\item $v_2$ is as close as desired to $u_i$ in the $\Cscr^1(M_i)$ topology.
\item $v_2^2=v_1^2$ on $M_{i+1}$.
\item $v_2^a>i$ on $\bigcup_{(j,k)\in I_a}\overline{\Omega_{j,k}\setminus D_{j,k}}$, $a=1,2$. 
Take into account {\rm (c3)}, \eqref{eq:hatv112}, and {\rm (d2)}.
\item $v_2^a>i+1$ on $\bigcup_{(j,k)\in I_a}\overline{\beta_{j,k}\setminus D_{j,k}}$, $a=1,2$.  
See {\rm (c4)}, \eqref{eq:hatv112},  and {\rm (d2)}.
\item $v_2^a>i+1$ on $\bigcup_{(j,k)\in I\setminus I_a} D_{j,k}$. 
Take into account {\rm (c5)}, \eqref{eq:hatv112}, \eqref{eq:hatv122}, and \eqref{eq:tau2}.
\item $\Flux_{v_1}(C)=\Flux_{u_i}(C)$ for any closed curve $C\subset M_i$. See {\rm (c6)}.
\end{enumerate}

Set $u_{i+1}=v_2\in\CMI_*(M_{i+1},\r^n)$. Since obviously
\[
	M_{i+1}\setminus M_i\subset \big(\bigcup_{(j,k)\in I_1}
	\overline{\Omega_{j,k}\setminus D_{j,k}}\big) \cup \big( \bigcup_{(j,k)\in I_1} D_{j,k}\big)
\]
and
\[
	bM_{i+1}\subset \big(\bigcup_{(j,k)\in I_1}
	\overline{\beta_{j,k}\setminus D_{j,k}}\big) \cup \big( \bigcup_{(j,k)\in I_1} D_{j,k}\big),
\]
{\rm (d3)} and {\rm (d5)} ensure condition {\rm (ii)} in the lemma, {\rm (d4)} and {\rm (d5)} give {\rm (iii)}, and {\rm (d6)}, 
Lemma \ref{lem:proper} {\rm (I)}, and the fact that $M_i$ is a strong deformation retract of $M_{i+1}$ imply {\rm (i)}. 
Taking into account {\rm (d1)}, this concludes the proof of the lemma in the noncritical case.

\smallskip
\noindent{\em The critical case}: Assume that $\rho$ has a critical point $p_{i+1}\in M_{i+1}\setminus M_i$. 

By the assumptions on $\rho$, $p_{i+1}$ is the only critical point of $\rho$ on $M_{i+1}\setminus M_i$ and 
it is a Morse point of Morse index either $0$ or $1$. 

Assume first that the Morse index of $p_{i+1}$ is $0$. In this case a new (simply connected) component 
of the sublevel set $\{\rho\leq r\}$ appears at $p_{i+1}$ when $r$ passes the value $\rho(p_{i+1})$. 
We then reduce the proof of the lemma to the 
noncritical case by defining $u_{i+1}=(u_{i+1}^1,\ldots,u_{i+1}^n)$ on this new component, $D$, 
as any nondegenerate conformal minimal immersion with $\max\{u_{i+1}^1,u_{i+1}^2\}>i+1$ on $D$.

Assume now that the Morse index of $p_{i+1}$ is $1$. In this case the change of topology of the sublevel set 
$\{\rho\leq r\}$ at $p_{i+1}$ is described by attaching to $M_i$ a smooth
arc $\gamma\subset \mathring M_{i+1}\setminus M_i$, and hence $M_i\cup \gamma$ is a 
Runge strong deformation retract 
of $M_{i+1}$. We assume without loss of generality that $M_i\cup \gamma$ is admissible in 
the sense of Definition \ref{def:admissible}. 
In view of Lemma \ref{lem:proper} {\rm (I)} and {\rm (II)}, we may extend $u_i$ to $M_i\cup\gamma$ 
as a nondegenerate generalized 
conformal minimal immersion 
$(\hat u_i=(\hat u_i^1,\ldots,\hat u_i^n),\hat f_i\theta)\in\GCMI_*(M_i\cup\gamma,\r^n)$ such that 
$\hat u_i=u_i$ on $M_i$, $\int_C \Im(f_i\theta) =\pgot(C)$ for all closed curve $C\subset M_i\cup\gamma$, and 
$\max\{\hat u_i^1,\hat u_i^2\}>i$ on $\gamma$. By Theorem \ref{th:Mergelyan} we find a Runge compact, smoothly 
bounded domain $\wt M_i$ and a nondegenerate conformal minimal immersion 
$v=(v^1,\ldots,v^n)\in\CMI_*(\wt M_i,\r^n)$ such that
\begin{itemize}
\item $M_i\cup\gamma \subset \mathring{\wt M_i}$ and $\wt M_i$ is a strong deformation retract of $M_{i+1}$,
\item $\Flux_v(C)=\pgot(C)$ for every closed curve $C\subset \wt M_i$, and
\item $\max\{v^1,v^2\}>i$ on $\wt M_i\setminus \mathring M_i$.
\end{itemize}
This reduces the proof to the noncritical case and proves the lemma.
\end{proof}

By recursively applying Lemma \ref{lem:proper} we may construct a sequence of nondegenerate 
conformal minimal immersions 
$\{u_i\in\CMI_*(M_i,\r^n)\}_{i\in\n}$ such that:
\begin{enumerate}[\rm (a)]
\item $u_i$ is as close to $u$ as desired in the $\Cscr^1(K)$ topology for all $i\in\n$.
\item $u_i$ is as close to $u_{i-1}$ as desired in the $\Cscr^1(M_{i-1})$ topology  for all $i\geq 2$.
\item $\Flux_{u_i}(C)=\pgot(C)$ for all closed curve $C\subset M_i$ and all $i\in\n$.
\item $\max\{u_{i+1}^1,u_{i+1}^2\}>i$ on $M_{i+1}\setminus M_i$ for all $i\in\n$.
\item $\max\{u_{i+1}^1,u_{i+1}^2\}>i+1$ on $bM_{i+1}$ for all $i\in\n$.
\end{enumerate}
Furthermore, if $n\geq 5$, applying Theorem \ref{th:desingBRS} at each step in the recursive construction 
we may assmue that
\begin{enumerate}[\rm (a)]
\item[\rm (f)] $u_i$ is an embedding for every $i\in\n$.
\end{enumerate}
 
If the approximations in {\rm (a)} and {\rm (b)} are close enough, then the limit 
$\tilde u=(\tilde u_1,\ldots,\tilde u_n):=\lim_{i\to\infty} u_i:M\to\r^n$ is a nondegenerate conformal minimal immersion 
satisfying the conclusion of the theorem. Indeed, {\rm (c)} trivially implies that $\Flux_{\tilde u}=\pgot$, whereas 
properties {\rm (d)} and {\rm (e)} ensure that $\max\{\tilde u_1,\tilde u_2\}:M\to\r$ and hence 
$(\tilde u_1,\tilde u_2):M\to\r^2$ 
are proper maps. Finally, if $n\geq 5$, $\tilde u$ can be taken an embedding; 
take into account {\rm (f)} and see for instance 
the proof of Theorem 4.5 in \cite{AL-C2} for a similar argument. 

This concludes the proof of Theorem \ref{th:proper2}.
\end{proof}


\subsection*{Acknowledgements}
A. Alarc\'on is supported by the Ram\'on y Cajal program of the Spanish Ministry of Economy and Competitiveness.
A.\ Alarc\'{o}n and F.\ J.\ L\'opez are partially supported by the MINECO/FEDER grants MTM2011-22547 and MTM2014-52368-P, Spain.  F.\ Forstneri\v c is partially  supported  by the research program P1-0291 and 
the grant J1-5432 from ARRS, Republic of Slovenia. 
Part of this work was made when F.\ Forstneri\v c visited the institute IEMath-Granada with support by 
the GENIL-SSV 2014 program.

We wish to thank an anonymous referee for the remarks which lead to improved presentation.


\vskip 0.2cm

\noindent Antonio Alarc\'{o}n

\noindent Departamento de Geometr\'{\i}a y Topolog\'{\i}a e Instituto de Matem\'aticas (IEMath-GR), Universidad de Granada, E--18071 Granada, Spain.

\noindent  e-mail: {\tt alarcon@ugr.es}

\vspace*{0.3cm}

\noindent Franc Forstneri\v c

\noindent Faculty of Mathematics and Physics, University of Ljubljana, and Institute
of Mathematics, Physics and Mechanics, Jadranska 19, SI--1000 Ljubljana, Slovenia.

\noindent e-mail: {\tt franc.forstneric@fmf.uni-lj.si}

\vspace*{0.3cm}

\noindent Francisco J.\ L\'opez

\noindent Departamento de Geometr\'{\i}a y Topolog\'{\i}a, Universidad de Granada, E--18071 Granada, Spain.

\noindent  e-mail: {\tt fjlopez@ugr.es}
\end{document}